\documentclass{amsart}

\usepackage{amssymb, amsthm, amssymb, enumerate, tikz, hyperref}
\hypersetup{
  colorlinks=true,
  linkcolor=blue!10!red,
}

\DeclareMathOperator*{\spt}{supp}

\newcommand{\R}{\mathbb{R}}
\newcommand{\C}{\mathbb{C}}
\newcommand{\N}{\mathbb{N}}
\newcommand{\A}{\mathcal{A}}
\newcommand{\xnice}{\mathbf{x}}

\newcommand{\Ho}{\mathcal{H}}

\newcommand{\CT}{\mathcal{C}}
\newcommand{\LT}{\mathcal{L}}
\newcommand{\fhat}{\hat{f}}
\newcommand{\ftilde}{\tilde{f}}
\newcommand{\ghat}{\hat{g}}

\newcommand{\Mtilde}{\tilde{M}}

\newcommand{\Mlog}{M_\mathrm{log}}
\newcommand{\MK}{M_{K}}
\newcommand{\MKtilde}{M_{\tilde{K}}}
\newcommand{\MKtildem}{M_{\tilde{K}_m}}
\newcommand{\Gtilde}{\tilde{G}}
\newcommand{\Ktilde}{\tilde{K}}
\newcommand{\inv}{^{-1}}
\newcommand{\obstacle}{\mathcal{O}}
\newcommand{\Ltwocomp}{L_\mathrm{c}^2}
\newcommand{\Ltwoloc}{L_\mathrm{loc}^2}
\newcommand{\Hcomp}[1]{H_\mathrm{c}^{#1}}
\newcommand{\Eloc}{E^\mathrm{loc}}
\newcommand{\Lin}{\mathcal{L}}
\newcommand{\ep}{\varepsilon}
\newcommand{\Ctilde}{\tilde{C}}
\newcommand{\Chat}{\hat{C}}
\newcommand{\buc}{\mathrm{BUC}(\R_+)}
\newcommand{\cab}{c_{\alpha,\beta}}

\newcommand{\abs}[1]{|#1|}
\newcommand{\absbig}[1]{\left|#1\right|}
\newcommand{\normbig}[1]{\left\| #1 \right\|}
\newcommand{\norm}[1]{\| #1 \|}

\newcommand{\cl}[1]{\overline{#1}}
\newcommand{\MKK}[1]{M_{#1}}

\newtheorem{thm}{Theorem}[section]
\newtheorem{prp}[thm]{Proposition}
\newtheorem{lem}[thm]{Lemma}
\newtheorem{cor}[thm]{Corollary}
\theoremstyle{definition}
\newtheorem{rem}[thm]{Remark}
\newtheorem{ex}[thm]{Example}

\numberwithin{equation}{section}

\allowdisplaybreaks

\begin{document}

\title[Local decay of $C_0$-semigroups]{Local decay of $C_0$-semigroups with a possible singularity of logarithmic type at zero}
\author[Reinhard Stahn]{Reinhard Stahn}

\begin{abstract} 
 We prove decay rates for a vector-valued function $f$ of a non-negative real variable with bounded weak derivative, under rather general conditions on the Laplace transform $\fhat$. This generalizes results of Batty-Duyckaerts (2008) and other authors in later publications. Besides the possibility of $\fhat$ having a singularity of logarithmic type at zero, one novelty in our paper is that we assume $\fhat$ to extend to a domain to the left of the imaginary axis, depending on a non-decreasing function $M$ and satisfying a growth assumption with respect to a different non-decreasing function $K$. The decay rate is expressed in terms of $M$ and $K$. We prove that the obtained decay rates are essentially optimal for a very large class of functions $M$ and $K$. Finally we explain in detail how our main result improves known decay rates for the local energy of waves on exterior domains.
\end{abstract}



\maketitle


\section{Introduction}\label{sec:intro}

In the last decade there has been much activity in the field of \emph{quantified} Tauberian theorems for $C_0$-semigroups, or more generally for functions of a non-negative real variable  \cite{Ch98,LiRa05,BaEnPrSch06,ChTo07,BaDu08,BoTo10,Ma11,BaChTo16,ChSe16,BaBoTo16,RoVe17}. See also \cite{Se15,Se16} and references therein for quantified Tauberian theorems on sequences and \cite{Ha17} for Dirichlet series. We refer to \cite{Ko2004} and \cite[Chapter 4]{ArBaHiNe2011} for a general overview on Tauberian theory.
 
Let $X$ be a Banach space and $f:\R_+\rightarrow X$ be a locally integrable function. For some continuous and non-decreasing function $M:\R_+\rightarrow (0,\infty)$ let us define
\begin{equation}\nonumber
 \Omega_M = \left\{ z\in \C; - \frac{1}{M(\absbig{\Im z})} < \Re z \leq 0 \right\} .
\end{equation}
The above mentioned articles impose essentially the Tauberian condition that the function $f$ has a bounded derivative $f'$ (in the weak sense), the Laplace transform $\fhat$ extends across the imaginary axis to $\Omega_M$ and it satisfies a growth condition, \emph{also expressed in terms of $M$}  in $\Omega_M$ at infinity. The decay rate (the rate of convergence to zero) is then determined in terms of $M$. For example, a polynomially growing $M$ yields (essentially) a polynomial decay rate and an exponentially growing $M$ yields a logarithmic decay rate. In general $\fhat$ could also have a finite number of singularities on the imaginary axis (see Martinez \cite{Ma11}), but we are not interested in this situation in the present article.
 
The pioneering works of Liu and Rao \cite{LiRa05} on the one side and Batkai, Engel, Pr\"uss, and Schnaubelt \cite{BaEnPrSch06} on the other side, focus on polynomial decay for orbits of $C_0$-semigroups. A generalization for functions (as formulated above) and to arbitrary decay rates was given by Batty and Duyckaerts \cite{BaDu08} for the first time. There the authors also improved the decay rates from \cite{LiRa05,BaEnPrSch06}. In the influential paper \cite{BoTo10} Borichev and Tomilov showed that the results of \cite{BaDu08} are optimal in the case of polynomial decay. We want to emphasize at this point that the main result of \cite{BaDu08} for the special case of a truncated orbit of a unitary group $U$ of operators (i.e. $f(t)=P_2 U(t) P_1$ for some bounded operators $P_1, P_2$) were already obtained in the earlier article \cite{PoVo99} by Popov and Vodev \emph{with the same rate of decay}. Actually the authors only formulated a theorem on polynomial decay but in the retrospective it is not difficult to generalize their proof to arbitrary decay rates. A major contribution to the field of Tauberian theorems is the recent article \cite{BaBoTo16} of Batty, Borichev, and Tomilov. The authors extended the known Tauberian theorems to $L^p$-rates of decay. On the basis of a technique already applied in \cite{BoTo10} the authors showed the optimality of their results in the case of polynomial decay. 

Another important observation, made in \cite{BaBoTo16}, concerns the above mentioned growth condition. In \cite{BaDu08} it was assumed that the norm of $\fhat(z)$ is bounded by $M(\absbig{\Im z})$ in $\Omega_M$. This condition was weakened by Borichev and Tomilov \cite{BoTo10} in case of polynomial decay, and later in \cite{BaBoTo16} assuming merely that $\fhat(z)$ can be bounded by a polynomial in $(1+\absbig{\Im z})$ and $M(\absbig{\Im z})$, i.e. there exist $C,\alpha,\beta\geq0$ such that
\begin{align}\label{eq:fhat_poly_K}
 \norm{z\fhat(z)} \leq C (1+\absbig{\Im z})^\alpha M(\absbig{\Im z})^\beta,\, z\in\Omega_M.
\end{align}
The factor $z$ is natural if one has an application to (local) decay of $C_0$-semigroups in mind. Moreover, it makes it easy to compare our results with those of others since often the rate of convergence of $\ghat(0)-\int_0^t g(s)ds$ is investigated for a bounded function $g$. Our results can be translated to this setting via $f'=g$ and vice versa. If $f'$ is bounded and if (\ref{eq:fhat_poly_K}) is satisfied, it is known that $\normbig{f(t)}=O(\Mlog\inv(ct))\inv,t\to\infty$ for a sufficiently small $c>0$ and $\Mlog(s)=M(s)\log(e\vee sM(s))$ (see e.g. \cite{BaBoTo16}). An inspection of the proofs in \cite{BaBoTo16} or \cite{BaDu08} reveals that this actually holds for $c\in(0,1/2)$ if $\alpha=0$ and $\beta=1$, or more generally for $c=(0,1)$ if $\Mlog$ was replaced by $s\mapsto M(s)\log(e\vee s^{2+\alpha}M(s)^\beta)$. See also a paper of Chill and Seifert \cite{ChSe16}, where admissible values for $c$ where explicitly discussed. Note that the value of $c$ has a very significant influence on the decay rate if $\Mlog$ grows at a sub-polynomial rate.

The aim of this paper is to generalize the above results in several directions. We illustrate the power and the need of our improvements in Section \ref{sec:waves_ext_dom}, where we apply our results to obtain \emph{semi-uniform} decay rates for the local energy of the wave equation on exterior domains. Such a problem can be reformulated as the question how fast a certain function $f$ decays. The wave equation on $d$-dimensional exterior domains reveals at least two weaknesses of the above abstract results. These weaknesses can prevent the above explained results to be (directly) applicable in this situation. For example if $d$ is even, it is well-known that $\fhat$ has a singularity of \emph{logarithmic type} at zero. This means that there exists a non-zero ($X$-valued) analytic function $\ftilde$ such that $z\mapsto\fhat(z)-\ftilde(z)\log(z)$ is analytic (more precisely, extends to an analytic function) in a neighbourhood of zero. Here by $\log:\C\backslash(-\infty,0]\to\C$ we denote a branch of the complex logarithm. For the particular case of the wave equation it is known that such a singularity has the effect, that the semi-uniform decay of the local energy can not be faster than $t^{-d}$ (see e.g. Vodev \cite{Vo99}). To the best of our knowledge we are the first to investigate such type of singularities in an abstract functional analytic setting. 

Ignoring the logarithmic singularity, which does not occur for $d$ being odd, the second weakness is that the condition (\ref{eq:fhat_poly_K}) seems to be too restrictive in this setting (in general). In fact, it can happen that $\fhat$ extends to a whole strip to the left of the imaginary axis (i.e. $M$ is constant) but the smallest (known) function $K$ for which
\begin{align}\label{eq:fhat_big_K}
 \norm{z\fhat(z)} \leq C K(\absbig{\Im z}),\, z\in \Omega_M
\end{align}
is satisfied, is of the form $K(s)=C\exp(Cs^\alpha),s\geq0$ for some $\alpha>0$ (we refer to Bony and Petkov \cite{BoPe06}). We solve this issue by significantly refining the proof presented in \cite{BaDu08} on a technical level. This proof is based on contour integrals and a fudge-factor argument due to Newman \cite{Ne80}. Our contribution to improve the method of Batty and Duyckaerts is to choose the contour and the fudge-factor in a very particular way which allows to consider functions $K$ in (\ref{eq:fhat_poly_K}) growing much faster than a polynomial in $(1+s)M(s)$. The decay rate we obtain is given by $\MKtilde\inv(ct)\inv$, for any $c\in(0,1)$ where $\MKtilde(s)=M(s)\log(e\vee sK(s)), s\geq0$. In some cases even $c=1$ is allowed. Our results can be \emph{directly} applied to the setting presented in \cite{BoPe06} and lead to improvements in their decay rates. We want to mention at this point, that the possibility of allowing such a more general growth bound in terms of $K$ was briefly discussed in \cite{BaBoTo16}. However, the authors gave no hint how the decay rate should look like or how the proof has to be modified.

For a large class of functions $M,K$ we show that our results are optimal. That is, given $c>1$ we construct functions $f:\R_+\to\C$, having bounded derivative, satisfying (\ref{eq:fhat_big_K}) for which $\liminf_{t\to\infty}\MKtilde\inv(ct)\absbig{f(t)}$ is bounded from below by a strictly positive constant. Our construction is based on a very similar construction, due to Borichev and Tomilov \cite{BoTo10}, showing that the ``logarithmic loss'' in the results of \cite{BaDu08} (the logarithmic term in the definition of $\Mlog$) can not be avoided if $M$ grows like $s^\alpha$ and $K$ like $s^\beta$ for some $\alpha>0, \beta>\alpha/2$. As in \cite{BoTo10} we show an analogous optimality result for the decay of $C_0$-semigroups.

Finally we want to point out an interesting side product of our research. In Section \ref{sec:Loc_dec_C0-SGs} we prove
\begin{thm}
 Let $M:\R_+\to(0,\infty)$ be a continuous non-decreasing function and $A$ the generator of a bounded $C_0$-semigroup $T$, which satisfies $\normbig{(is-A)\inv}\leq M(\absbig{s}), s\in\R$. For any $c\in(0,1)$
 \begin{equation}\nonumber
  \normbig{T(t)A\inv} = O\left( \frac{1}{\Mlog\inv(ct)} \right), t\rightarrow\infty .
 \end{equation}
 Here $\Mlog(s)=M(s)\log(e \vee s M(s)), s\geq0$.
\end{thm}
Of course this is essentially a well known result due to Batty and Duyckaerts \cite{BaDu08}, but note that here the range of admissible values for $c$ is twice as large as in all proofs known to us (see discussion above). Moreover, it can be shown that - in this generality - the theorem would be false if $c>1$ were allowed. We refer the reader to the end of Section \ref{sec:Loc_dec_C0-SGs} for this fact. It is an open problem whether the theorem remains true for $c=1$.

The paper is organized as follows. In Section \ref{sec:main} we prove the main result (Theorem \ref{thm:main}) of the paper and establish decay rates for functions with bounded derivative under assumptions on the Laplace transform. Although it is not the main objective of this paper we also prove a result on $L^p$-rates of decay, generalizing \cite[Theorem 4.1]{BaBoTo16} (see Theorem \ref{thm:main_Lp}). In Section \ref{sec:Loc_dec_C0-SGs} we deduce (local) decay results for $C_0$-semigroups. Section \ref{sec:optimality} is devoted to investigations concerning optimality of Theorem \ref{thm:main}. Finally in Section \ref{sec:waves_ext_dom} we apply our results to a wave equation on an exterior domain. For those being not familiar with this setting we describe in detail the most important peculiarities of waves on exterior domains - from the functional analytic point of view. 

\textbf{Notation:} We write $\R_+=[0,\infty)$, $\R_-=-\R_+$ and $\C_+=\{z\in\C; \Re z>0\}$. Given two real numbers $a$ and $b$ we denote $a\vee b = \max\{a,b\}$ and $a\wedge b = \min\{a,b\}$. We use the widely-used convention that whenever a constant called $C$ appears multiple times in a chain of (in)equalities it does not necessarily have the same value each time. Throughout this paper, given two functions $M,K:\R_+\rightarrow \R$ we define the function $\MK:\R_+\rightarrow\R$ by $\MK(s)=M(s)\log(e \vee K(s))$. Moreover, if $M$ assumes only strictly positive values we define $\Omega_M=\{z\in\C; -1/M(\absbig{\Im z}) < \Re z \leq 0\}$. Given $m\in\N_1$ and a function $K:\R_+\to(0,\infty)$ we define functions $K_m, K_{m,\log}:\R_+\to\R_+$ by $K_m(s)=s^mK(s)$ and $K_{m,\log}(s)=s^m\log(e\vee s)K(s)$.

For a measurable and exponentially bounded function $f:\R_+\to X$ with values in a Banach space $X$, we denote its Laplace transform by $\fhat$. Recall that $\fhat(z)=\int_0^\infty e^{-zt} dt$ exists as an absolutely convergent integral on a right half-plane $H_\omega=\{z\in\C;\Re z>\omega\}$ for a sufficiently large $\omega$. If $\fhat$ extends analytically to a connected domain, containing this half plane we still denote the resulting function by $\fhat$. Observe that a locally integrable function with bounded derivative is exponentially bounded because it can increase at most at a linear rate. Thus, the Laplace transform is absolutely convergent in $H_0$.

Let $K:\R_+\to(0,\infty)$ be a function. We say that $K$ has \emph{positive increase} (of index $a\geq0$) if there exist constants $C_a\geq 1$ and $s_a>0$ such that for all $s_a\leq s \leq R$
\begin{align}\label{eq:pos_inc}
  s^{-a} K(s) \leq C_{a} R^{-a}K(R) .
\end{align}
$K$ has positive increase of index $a+$ if it has positive increase of index $a+\ep$ for some $\ep>0$. If we do not specify the index, we mean for some \emph{strictly positive} index. Note that any non-decreasing function has positive increase of order $0$.

In Remarks \ref{rem:Ctilde_in_main}, \ref{rem:Ctilde_in_main_C0_local} and the formulation of Corollaries \ref{cor:main_C0_strip}, \ref{cor:main_C0} we make use of terminology involving partially ordered sets. Some ``constants'' in these remarks and theorems are actually \emph{non-decreasing functions} defined on cartesian products of certain partially ordered sets. The ordering on the cartesian product $S_1\times\ldots\times S_n$ of partially ordered sets $S_1,\ldots,S_n$ is defined via $(s_1,\ldots,s_n)\leq(t_1,\ldots,t_n):\Leftrightarrow \forall j: s_j\leq t_j$. Any space of real valued functions is partially ordered by the pointwise ordering. Given a partially ordered set $S$, a function $C:S\to[0,\infty)$ is called \emph{non-decreasing} if $s_1\leq s_2$ implies $C(s_1)\leq C(s_2)$ for any $s_1,s_2\in S$. The term \emph{non-increasing} is defined analogously, with the second inequality sign being reversed. We say that $C$ is \emph{purely numerical} if $C$ is a constant, i.e. there exists $C_0\in[0,\infty)$ such that $C(s)=C_0$ for all $s\in S$. We use the latter notion to express that certain constants in the formulation of a theorem actually do not depend on any objects occurring in the assumptions of the theorem. 


\section{The main result}\label{sec:main}
This section is devoted to the heart of our paper. We establish decay rates for a vector-valued function with bounded derivative under assumptions on the Laplace transform. Throughout the section $(X,\norm{\cdot})$ denotes a Banach space.
\begin{thm}\label{thm:main}
 Let $f:\R_+\rightarrow X$ be a locally integrable function such that its weak derivative $f^{(m)}$ of order $m\in \N_1$ is bounded. Let $M, K:\R_+\rightarrow (0,\infty)$ be continuous and non-decreasing functions such that for some $\ep\in(0,1)$ and $r_1,C_\ep>0$
 \begin{equation}\label{eq:K_vs_M}
  K(s) \leq C_\ep e^{e^{(sM(s))^{1-\varepsilon}}},\, s\geq r_1.
 \end{equation}
 Assume that for some $r\geq M(0)$ and some analytic function $\ftilde:B_r\rightarrow X$ the mapping $z\mapsto \fhat(z) - \ftilde(z)\log(z)$ is analytic on $B_r$. Assume furthermore that $\fhat$ extends analytically to $(\Omega_M\cup\C_+)\backslash\R_-$, continuously to $(\overline{\Omega_M}\cup\C_+)\backslash\R_-$ and that for some $\Chat>0$
 \begin{align}\label{eq:fhat_at_zero}
  \norm{z\fhat(z)} \leq \Chat K(\absbig{\Im z}) \text{ for all } z\in \Omega_M,\, \abs{\Im z}>r_1 .
 \end{align}
 For any $n\in\N_1$ there exist constants $C_m, t_m > 0$ such that 
 \begin{align}\label{eq:decay_rate}
  \normbig{f(t) + \ftilde_{n-1}\left( \frac{d}{dt} \right)t^{-1}}  \leq \sup_{-r<x<0} \frac{\| \ftilde^{(n)}(x) \|}{t^{n+1}} + \frac{C_m}{\MKtilde^{-1}(t)^m}
 \end{align}
 for all $t\geq t_m$, where $\ftilde_{n-1}$ is the Taylor polynomial of $\ftilde$ up to order $n-1$, and $\Ktilde=K_{m,\log}$. If $K$ has positive increase the theorem remains true for $\Ktilde=K_m$.
\end{thm}

\begin{rem}\label{rem:Ctilde_in_main}
 Sometimes it is desirable to know the dependencies of $C_m$ and $t_m$ on the hypotheses. Let us define $\norm{\ftilde}_{(-r,0)}=\sup\{\norm{f(x)}; x\in(-r,0)\}$, $\norm{\ftilde}_n=\max\{\norm{\ftilde^{(j)}(0)}; j\leq n-1\}$ and $\norm{f}_m = \max\{\norm{f^{(j)}(0)}; j\leq m-1\}$. Furthermore, let $\Chat_0>0$ be a constant such that
 \begin{align*}
  \sup\{\norm{z\fhat(z)}; z\in\Omega_M\backslash\R_-, \abs{\Im z}<r_1\} \leq \Chat_0 .
 \end{align*}
 In the following $C'\geq1$ (respectively $C'_0\geq 1$) describes a constant which can be seen as a non-decreasing function in the variables $1/M(r_1), \norm{f}_m/\Chat K(r_1)$ (respectively $1/M(r_1),\norm{f}_m/\Chat_0$). We refer the reader to the end of the introduction for an explanation what we precisely mean by \emph{non-decreasing} with respect to several variables. Now $t_m$ can be chosen to be a non-decreasing function depending on $m,1/K,M(r_1)$, $\ep\inv,r_1,r_1\inv,C'_0\Chat_0/C'\Chat$. With this choice of $t_m$ we can choose $C_m$ to be of the form
 \begin{align*}
  C_m=C'_m(\norm{f^{(m)}}_{\infty}+C'\Chat) + C''_m(\norm{\ftilde}_{(-r,0)}+\norm{\ftilde}_n)
 \end{align*}
 with $C'_m>0$ (respectively $C''_m$) being a non-decreasing function depending on $m,r_1\inv,\ep\inv,C_\ep$ from (\ref{eq:K_vs_M}) (respectively $m,n,\ep\inv,1/M(r_1),1/K(r_1)$). If the information about positive increase is used, $t_m$ also depends on $s_a$ and $C'_m$ also depends on $a\inv, C_a$ from (\ref{eq:pos_inc}). All this will be pointed out in the proof.
\end{rem}

For the readers convenience we give some examples of decay rates $R(t)=\MKtilde\inv(t)$ for a variety of possible choices for $M$ and $K$. We fix $m=1$ for simplicity. Let $\alpha,\alpha',\delta,\delta'\in(0,\infty), \beta\in[0,\infty), \gamma\in(0,1)$. We write $R(t)\lesssim R_1(t)$ if the left-hand side is estimated by a constant times the right-hand side for large $t$. We write $R(t)\approx R_1(t)$ if $R(t)\lesssim R_1(t)$ and the reverse inequality hold.

\emph{Exponential decay.}
\begin{itemize}
 \item[(a)] $M(s)=\delta\inv, K(s)\approx 1\vee s^\alpha$. Then $R(t) \approx \exp(\delta t/(1+\alpha))$.
\end{itemize}

\emph{Super-polynomial but sub-exponential decay.}
\begin{itemize}
 \item[(b)] $M(s)= \delta\inv\log(e\vee s)^\alpha$ and $K$ grows at most at a sub-polynomial rate. Then $\exp((c_1 \delta t)^{1/(1+\alpha)}) \lesssim R(t) \lesssim \exp((\delta t)^{1/(1+\alpha)})$ for all $c_1\in(0,1)$.
 \item[(c)] $M(s)= \delta\inv\log(e\vee s)^\alpha$, $K(s)\approx 1\vee s^{\alpha'}$. \\ Then $R(t) \approx \exp((\delta t/(1+\alpha'))^{1/(1+\alpha)})$.
 \item[(d)] $M(s)= \delta\inv \log(e\vee s)^\alpha$ and $K(s)\approx\exp(\delta'^{-1} \log(e\vee s)^{1+\alpha'})$. Then \\ $\exp((c_1 \delta\delta' t)^{1/(1+\alpha+\alpha')}) \lesssim R(t) \lesssim \exp((\delta\delta' t)^{1/(1+\alpha+\alpha')})$ for all $c_1\in(0,1)$.
\end{itemize}
 In (d) one can even take $c_1=1$ if $\alpha'\geq 1$.
 
\emph{Polynomial decay.}
\begin{itemize}
 \item[(e)] $M(s)\approx 1\vee s^\beta, K(s)\approx \exp(C s^{\alpha})$. Then $R(t)\approx t^{1/(\alpha+\beta)}$.
 \item[(f)] $M(s)\approx \log(e\vee s)^{\alpha'}, K(s)\approx \exp(C s^{\alpha})$. Then $R(t)\approx (t/\log(t)^{\alpha'})^{1/\alpha}$.
 \item[(g)] $M(s)\approx 1\vee s^\alpha, K(s)\approx 1\vee s^{\alpha'}$. Then $R(t)\approx (t/\log(t))^{1/\alpha}$.
\end{itemize}
 Note that in (e) we allow $\beta=0$. This is an important special case since it is relevant for the application presented in Section \ref{sec:waves_ext_dom}.
 
\emph{Logarithmic decay.}
\begin{itemize}
 \item[(h)] $1\lesssim M(s)\lesssim 1\vee s^\alpha$, $K(s)\approx\exp(\exp(Cs^{\gamma}))$. Then $R(t)\approx \log(t)^{1/\gamma}$.
 \item[(i)] $M(s) \approx \exp(Cs^{\alpha})$, $K(s)\approx \exp(Cs^{\alpha'})$. Then $R(t)\approx \log(t)^{1/\alpha}$.
 \item[(j)] $M(s) \approx \exp(Cs^{\alpha})$, $K(s)\approx \exp(\exp(Cs^{\alpha'}))$. Then $R(t)\approx \log(t)^{1/(\alpha+\alpha')}$.
\end{itemize}
 If in (h) we also have $M(s)\gtrsim 1\vee s^\alpha$ we could allow any $\gamma\in(0,1+\alpha)$. We remark here, that the known results from the literature cannot treat the cases (d), (e), (f), (h), (i) (here assuming $\alpha'>\alpha$) and (j). Our results concerning (a), (b) and (c) are sharper than known results, since in the literature the decay rate is essentially estimated by $R(ct)$ for a certain $c<1$ depending on $\alpha, \alpha'$, as discussed in the introduction. Excluding (b) and with only very minor restrictions in cases (g) and (i), in Theorem \ref{thm:optimality}, we prove the optimality of the obtained decay rate in a sense to be made precise in that theorem.

\begin{proof}[Proof of Theorem \ref{thm:main}]
For simplicity we assume $m=1$ and write $\Ktilde=K_{1,\log}$. At the very end of the proof we briefly explain the modification of the proof which leads to the conclusion of the theorem in case $m>1$. Let $k$ be a strictly positive natural number to be fixed later. We define the function $\psi:\C\backslash\{-i,+i\}\rightarrow\C$ by
\begin{align*}
 \psi(z) = c_k \exp \left( - \exp \left( \left(\frac{2}{1+z^2}\right)^k \right) \right), \text{ where } c_k=e^{e^{2^k}}.
\end{align*}

Let $R>0$ be a natural number to be chosen later (depending on $t$). Depending on $R$ and an additional parameter $\delta\in(0,r_1)$ we define now various contours for integration in the complex plane. Therefore we abbreviate $y_R=R-R(RM(R))^{-1/(k+2)}$. Clearly there is an $R_0\geq 0$, solely depending on $M,k$ and $r_1$ such that $y_R>r_1$ for all $R\geq R_0$. In the following we assume $R\geq R_0$.

\begin{align*}
 \gamma_{11} &= \{ R(x-i(1-x^{\frac{1}{k+2}}));\, x\in(0,1)\},  \\
 \gamma_{12} &= \{ R((1-x)+i(1-(1-x)^{\frac{1}{k+2}}));\, x\in(0,1)\},  \\
 \gamma_{21} &= \{ -R(x-i(1-x^{\frac{1}{k+2}}));\, x\in(0,1)\}, \\
 \gamma_{22} &= \{ -R((1-x)+i(1-(1-x)^{\frac{1}{k+2}}));\, x\in(0,1)\}, \\
 \gamma_{31} &= \{ -R(x-i(1-x^{\frac{1}{k+2}}));\, x\in(0,(RM(R))^{-1})\}, \\
 \gamma_{32} &= \{ - M(R)^{-1}+i(y_R-y);\, y\in(0,y_R-r_1) \}, \\
 \gamma_{33} &= \{ (1-\theta)(ir_1 - M(R)\inv) + \theta(i\delta - M(r_1)\inv) ; \theta \in (0,1) \} \\
 \gamma_{34} &= \{ x+i\delta;\, x\in(- M(r_1)^{-1},0)\}\cup
                \{(i\delta\cos\varphi,-i\delta\sin\varphi);\, \varphi\in(-\frac{\pi}{2},\frac{\pi}{2})\} \\ 
                &\phantom{= } \cup\{ -x-i\delta;\, x\in(0, M(r_1)^{-1})\} \\
 \gamma_{35} &= \{ (1-\theta)(-i\delta - M(r_1)\inv) + \theta( -ir_1 - M(R)\inv) ; \theta \in (0,1) \} \\
 \gamma_{36} &= \{ - M(R)^{-1}-iy;\, y\in(r_1, y_R)\}, \\
 \gamma_{37} &= \{ -R((1-x)+i(1-(1-x)^{\frac{1}{k+2}}));\, x\in(1-(RM(R))^{-1},1)\}.
\end{align*}

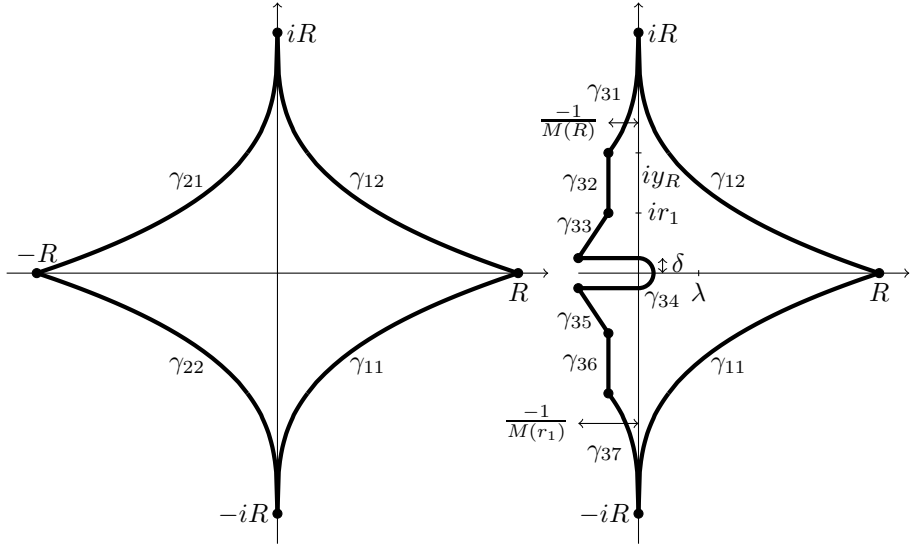
\begin{center}
  \begin{tikzpicture}[scale=0.4]
    \begin{scope}[shift={(0,0)}] 
      \draw[->, thin] (-9,0) -- (9,0) ; 
      \draw[->, thin] (0,-9) -- (0,9) ; 
      \draw[ultra thick, domain=0:1/8] plot ( 8*\x, {-8*(1-pow(\x,1/3))} ); 
      \draw[ultra thick, domain=1/8:1] plot ( 8*\x, {-8*(1-pow(\x,1/3))} ); 
      \draw[ultra thick, domain=1/8:1] plot ( 8*\x, {8*(1-pow(\x,1/3))} ); 
      \draw[ultra thick, domain=0:1/8] plot ( 8*\x, {8*(1-pow(\x,1/3))} ); 
      \draw[ultra thick, domain=1/8:1] plot ( -8*\x, {8*(1-pow(\x,1/3))} ); 
      \draw[ultra thick, domain=0:1/8] plot ( -8*\x, {8*(1-pow(\x,1/3))} ); 
      \draw[ultra thick, domain=0:1/8] plot ( -8*\x, {-8*(1-pow(\x,1/3))} ); 
      \draw[ultra thick, domain=1/8:1] plot ( -8*\x, {-8*(1-pow(\x,1/3))} ); 
      \draw[fill] (8,0) circle [radius=0.15]; \node[below] at (8,0) {$R$};
      \draw[fill] (0,8) circle [radius=0.15]; \node[right] at (0,8) {$iR$};
      \draw[fill] (-8,0) circle [radius=0.15]; \node[above] at (-8,0) {$-R$};
      \draw[fill] (0,-8) circle [radius=0.15]; \node[left] at (0,-8) {$-iR$};
      \node[below] at (3,-2.5) {$\gamma_{11}$};
      \node[above] at (3,2.5) {$\gamma_{12}$};
      \node[above] at (-3,2.5) {$\gamma_{21}$};
      \node[below] at (-3,-2.5) {$\gamma_{22}$};
    \end{scope}
    \begin{scope}[shift={(12,0)}] 
      \draw[->, thin] (-2,0) -- (9,0) ; 
      \draw[->, thin] (0,-9) -- (0,9) ; 
      \draw[ultra thick, domain=0:1/8] plot ( 8*\x, {-8*(1-pow(\x,1/3))} ); 
      \draw[ultra thick, domain=1/8:1] plot ( 8*\x, {-8*(1-pow(\x,1/3))} ); 
      \draw[ultra thick, domain=1/8:1] plot ( 8*\x, {8*(1-pow(\x,1/3))} ); 
      \draw[ultra thick, domain=0:1/8] plot ( 8*\x, {8*(1-pow(\x,1/3))} ); 
      \draw[ultra thick, domain=0:1/8] plot ( -8*\x, {8*(1-pow(\x,1/3))} ) ; 
      \draw[ultra thick] (-1,4) -- (-1,2) -- (-2,0.5) -- (0,0.5); 
      \draw[ultra thick] (0,-0.5) arc [radius=0.5, start angle = -90, end angle = 90]; 
      \draw[ultra thick] (0,-0.5) -- (-2,-0.5) -- (-1,-2) -- (-1,-4); 
      \draw[ultra thick, domain=0:1/8] plot ( -8*\x, {-8*(1-pow(\x,1/3))} ); 
      \draw[fill] (8,0) circle [radius=0.15]; \node[below] at (8,0) {$R$};
      \draw[fill] (0,8) circle [radius=0.15]; \node[right] at (0,8) {$iR$};
      \draw[fill] (0,-8) circle [radius=0.15]; \node[left] at (0,-8) {$-iR$};
      \draw[fill] (-1,4) circle [radius=0.15];
      \draw[fill] (-1,2) circle [radius=0.15];
      \draw[fill] (-1,-4) circle [radius=0.15];
      \draw[fill] (-2,0.5) circle [radius=0.15];
      \draw[fill] (-2,-0.5) circle [radius=0.15];
      \draw[fill] (-1,-2) circle [radius=0.15];
      \draw[thin] (-0.1,2) -- (0.1,2); \node[right] at (0,2) {$ir_1$};
      \draw[thin] (-0.1,4) -- (0.1,4); \node[right] at (-0.2,3.2) {$iy_R$};
      \draw[<->, thin] (0.8,0) -- (0.8,0.5); \node[right] at (0.8,0.35) {$\delta$};
      \draw[thin] (2,-0.1) -- (2,0.1); \node[below] at (2,0) {$\lambda$};
      \draw[<->, thin] (-1,5) -- (0,5); \node[left] at (-1,5) {$\frac{-1}{M(R)}$};
      \draw[<->, thin] (-2,-5) -- (0,-5); \node[left] at (-2,-5) {$\frac{-1}{M(r_1)}$};
      \node[below] at (3,-2.5) {$\gamma_{11}$};
      \node[above] at (3,2.5) {$\gamma_{12}$};
      \node[left] at (-0.2,6) {$\gamma_{31}$};
      \node[left] at (-1,3) {$\gamma_{32}$};
      \node[above] at (-2.1,1) {$\gamma_{33}$};
      \node[below] at (0.8,-0.3) {$\gamma_{34}$};
      \node[below] at (-2.1,-1) {$\gamma_{35}$};
      \node[left] at (-1,-3) {$\gamma_{36}$};
      \node[left] at (-0.2,-6) {$\gamma_{37}$};
    \end{scope}
  \end{tikzpicture}
\end{center}

Since we plan to consider the limit $\delta\downarrow0$ we may assume that none of the contours intersects another one. If we have to use a parametrization of one of these contours we do it via $x,y,\theta$ or $\varphi$ as indicated in the definitions of the contours. This also determines an orientation of the paths. Moreover, we define $\gamma_1=\gamma_{11}+\gamma_{12}$, $\gamma_2=\gamma_{21}+\gamma_{22}$ and $\gamma_3=\gamma_{31}+\ldots+\gamma_{37}$. Note that $\gamma_1+\gamma_2$ and $\gamma_1+\gamma_3$ are closed paths encircling each of the points from the interval $(\delta,R)$. Also note that the derivative of the parametrization of any of the above paths approaching $+iR$ or $-iR$ can be estimated by a constant times $Rx^{-1}$ or $R(1-x)\inv$, depending on whether $x=0$ or $x=1$ corresponds to the point $\pm iR$.

Now let us define the bounded function $g:\R\rightarrow X$ via $g(t)=-f'(t)$ for positive $t$ and extend it by $0$ for negative arguments. Observe that $\ghat(z)=-z\fhat(z)+f(0)$. Without loss of generality we may assume that $f(0)=0$, otherwise we could replace $f$ by $f_0=f-f(0)\chi$ where $\chi\in C_c^{\infty}((-1,1))$ satisfies $\chi(0)=1$. Note that for example $\norm{z\fhat_0(z)}\leq C'\Chat K(\abs{\Im z}), z\in\Omega_m, \abs{\Im z}\geq r_1$ for a constant $C'$ as in Remark \ref{rem:Ctilde_in_main}. Let us define the function $h_t$ on $\C_+$ by 
\begin{align*}
 h_t(z) = \ghat(z) - \int_0^t e^{-zs} g(s) ds .
\end{align*}
By assumptions, $\ghat$ and $h_t$ extend to analytic functions on $(\Omega_M\cup \C_+)\backslash\R_-$, which are continuous on $(\cl{\Omega_M}\cup \C_+)\backslash\R_-$. Observe that $f(t)=h_t(0)$, if we extend $\ghat$ continuously by $0$ at $0$. Therefore
\begin{align*}
 f(t) &= \lim_{\lambda\downarrow0} h_t(\lambda) \psi(R^{-1}\lambda) e^{\lambda t} \\
 &= \lim_{\lambda\downarrow0} \lim_{\delta\downarrow0} \frac{1}{2\pi i}
    \int_{\gamma_1+\gamma_3} \psi(R^{-1}z) h_t(z) e^{zt} \frac{dz}{z-\lambda} \\
 &= \lim_{\lambda\downarrow0} \frac{1}{2\pi i}
    \int_{\gamma_1} \psi(R^{-1}z) \left( \ghat(z) - \int_0^t e^{-zs} g(s) ds \right) e^{zt} \frac{dz}{z-\lambda} \\
 &+ \lim_{\lambda\downarrow0} \frac{1}{2\pi i}
    \int_{\gamma_2} \psi(R^{-1}z) \left( - \int_0^t e^{-zs} g(s) ds \right) e^{zt} \frac{dz}{z-\lambda} \\
 &+ \lim_{\lambda\downarrow0} \lim_{\delta\downarrow0} \frac{1}{2\pi i}
    \int_{\gamma_3} \psi(R^{-1}z) \ghat(z) e^{zt} \frac{dz}{z-\lambda} \\
 &=: I_1 + I_2 + I_3 .
\end{align*}
Actually at the moment we do not know whether the integrals above really exist since $\psi$ has (essential) singularities at $\pm i$. However, the following lemma fixes this problem. It implies that $\psi(R\cdot)$ is bounded on all the contours and decays fast enough (for our purposes) as $z$ approaches $\pm iR$ along $\gamma_1,\gamma_2$ or $\gamma_3$. Thus - in the spirit of Newman \cite{Ne80} - our $\psi$ serves as a ``fudge factor'' in our Cauchy integrals.
\begin{lem}\label{lem:psi_near_poles}
 Let $\varepsilon\in(0,1)$ and $k\in\N_1$ with $k > 2\varepsilon^{-1} - 2$. There exists $C>0$, solely depending on $k$ such that
 \begin{align*}
  \absbig{\psi(z)} \leq C \exp(-\exp( x^{-(1-\varepsilon)} ))
 \end{align*}
 for all $z\in\C$ which can be represented as
 \begin{align*}
  z = x + i(1-y) \text{ or } z = x + i (-1+y)
 \end{align*}
 where $y\in(0,1)$ and  $\absbig{x} = y^{k+2}$.
\end{lem}
\begin{proof}
 By symmetry of the function $\psi$ it suffices to consider the case where $z$ can be represented as $z = x + i(1-y)$ for $y\in(0,1)$ and  $\absbig{x} = y^{k+2}$. Clearly $\psi$ is bounded if $z$ stays away from $i$. Thus it suffices to consider the asymptotic behaviour of $\psi(z)$ as $y$ approaches $0$. We have $x=o(y)$ as $y\rightarrow0$ and 
 \begin{align*}
  \frac{1}{1+z^2} = \frac{a-ib}{a^2+b^2} \text{ with } a&=y(2-y)+x^2=(2+o(1))y \\ \text{ and } b&=2x(1-y)=(2+o(1))x.
 \end{align*}
 Therefore a short calculation yields
 \begin{align*}
  \left( \frac{2}{1+z^2} \right)^k = (1+o(1))y^{-k} + io(1)
 \end{align*}
 as $y\downarrow0$. Here and in the following $o(1)$ replaces \emph{real valued} terms converging to zero as $y\downarrow0$. The last line in turn implies
 \begin{align*}
  \Re \exp \left( \left( \frac{2}{1+z^2} \right)^k \right) = e^{(1+o(1))y^{-k}}.
 \end{align*}
 This yields the claim.
\end{proof}
\textit{Estimation of $I_1$.} By dominated convergence we can perform the limit $\lambda\downarrow0$ by simply setting $\lambda=0$ in the integral.  We further split the integral $I_1=I_{11}+I_{12}$ according to the decomposition of the path $\gamma_1=\gamma_{11}+\gamma_{12}$. Using $\absbig{\dot{\gamma}_{11}}(x)\leq C R x^{-1}$ and assuming $k>2\ep\inv-2$ we get
\begin{align}\label{eq: I11 modify proof for large m}
 \normbig{I_{11}} &= \normbig{ \int_0^1 \int_t^{\infty} \underbrace{\psi(R^{-1}\gamma_{11}(x))}_{= o(x^{\infty}) \text{ by Lemma \ref{lem:psi_near_poles}}} e^{-(s-t)\gamma_{11}(x)} g(s) ds  \frac{\dot{\gamma}_{11}(x)}{\gamma_{11}(x)} dx} \\ \nonumber
 &\leq C \int_{-\infty}^{\infty} \underbrace{\int_0^1 e^{-sRx}x dx}_{\leq C(1+Rs)^{-2}} \normbig{g(t+s)} ds \\ \nonumber
 &\leq \frac{C}{R} P_{R^{-1}}*\normbig{g}(t) \leq C \frac{\norm{f'}_{\infty}}{R}.
\end{align}
Here $P_y:\R\to(0,\infty)$ for $y>0$ is the Poisson kernel which is given by $P_y(t)=y/(\pi(t^2+y^2))$. We have proved that
\begin{align}\label{eq: estimate on I1}
 I_1 \leq \frac{C}{R} P_{R^{-1}}*\normbig{g}(t) \leq C \frac{\norm{f'}_{\infty}}{R}
\end{align}
since the estimation of $I_{12}$ is analogous.

\textit{Estimation of $I_2$.} This is almost the same procedure as in the estimation of $I_1$. Again we can perform the limit $\lambda\downarrow0$ by simply setting $\lambda=0$ in the integral and we split the integral $I_2=I_{21}+I_{22}$ according to the decomposition of the path $\gamma_2=\gamma_{21}+\gamma_{22}$.
\begin{align*}
 \normbig{I_{21}} &= \normbig{ \int_0^1 \int_0^t \underbrace{\psi(R^{-1}\gamma_{21}(x))}_{= o(x^{\infty}) \text{ by Lemma \ref{lem:psi_near_poles}}} e^{-(t-s)\gamma_{21}(x)} g(s) ds  \frac{\dot{\gamma}_{21}(x)}{\gamma_{21}(x)} dx }  \\
 &\leq C \int_{-\infty}^{\infty} \underbrace{\int_0^1 e^{-sRx}x dx}_{\leq C(1+Rs)^{-2}} \normbig{g(t-s)} ds \\
 &\leq \frac{C}{R} P_{R^{-1}}*\normbig{g}(t) \leq C \frac{\norm{f'}_{\infty}}{R}.
\end{align*}
Again the estimation of $I_{22}$ is analogous and we have therefore proved
\begin{align}\label{eq: estimate on I2}
 I_2 \leq \frac{C}{R} P_{R^{-1}}*\normbig{g}(t) \leq C \frac{\norm{f'}_{\infty}}{R}.
\end{align}
Note that $C$ in (\ref{eq: estimate on I1}) and (\ref{eq: estimate on I2}) solely depends on $k$ since $\psi$ and (more importantly) the path of integration solely depend on $k$ and on $R$.

\textit{Estimation of $I_3$.} We split the integral $I_3=I_{31}+\ldots+I_{37}$ according to the decomposition of the path $\gamma_3=\gamma_{31}+\ldots+\gamma_{37}$. It suffices to investigate $I_{34}, I_{35}, I_{36}$ and $I_{37}$ since the estimation of $I_{31}, I_{32}$ and $I_{33}$ is similar to the estimation of $I_{37}, I_{36}$ and $I_{35}$. Therefore, without loss of generality we can ignore $I_{31}, I_{32}$ and $I_{33}$ in the following. By dominated convergence we can perform the limits $\delta\downarrow0$ and $\lambda\downarrow0$ by simply setting $\delta=\lambda=0$ in the integrals $I_{37},I_{36}$ and $I_{35}$. The limits in the integral $I_{34}$ are performed later on.

Let us fix $k=\lceil4\varepsilon^{-1}-2\rceil$, where $\varepsilon$ is as in (\ref{eq:K_vs_M}), and recall Lemma \ref{lem:psi_near_poles}. Since $\abs{\dot{\gamma}_{37}}\leq CR(1-x)\inv$ for a $C$ solely depending on $\ep\inv$ we get
\begin{align}\nonumber
 \normbig{I_{37}} &\leq \int_{1-(RM(R))^{-1}}^1 \absbig{\psi(R^{-1}\gamma_{35}(x))} \normbig{\ghat(\gamma_{35}(x))} \absbig{\dot{\gamma}_{35}(x)} \frac{dx}{\absbig{\gamma_{35}(x)}} \\ \nonumber
 &\leq C\Chat \int_{1-(RM(R))^{-1}}^1 (1-x)\inv \exp(-\exp((1-x)^{-(1-\varepsilon/2)})) K(R) dx \\ \label{eq: estimate on I37}
 &\leq C\Chat K(R) \exp(-\exp((RM(R))^{-(1-\varepsilon)})) \leq C \frac{\Chat}{R}.
\end{align}
In (\ref{eq: estimate on I37}) the constant $C$ depends on $\ep\inv$ as well as on $C_\ep$ since we used (\ref{eq:K_vs_M}) in the last inequality.
Without loss of generality we may assume that $R_0\geq e > 1$. Thus, for all $R\geq R_0$ we get
\begin{align}\nonumber
 \normbig{I_{36}} &\leq \int_{r_1}^{y_R} \absbig{\psi(R^{-1}\gamma_{36}(y))} \normbig{\ghat(\gamma_{64}(y))} e^{-\frac{t}{M(R)}} \absbig{\dot{\gamma}_{36}(y)} \frac{dy}{\absbig{\gamma_{36}(y)}} \\ \nonumber
 &\leq C \Chat \int_{r_1}^R (r_1+y)\inv K(y) e^{-\frac{t}{M(R)}} dy \\ \label{eq: estimate on I36}
 &\leq C \Chat \log(R) K(R) e^{-\frac{ t}{M(R)}} .
\end{align}
Here $C$ depends on $\ep\inv$ and $r_1\inv$. Note that in (\ref{eq: estimate on I36}) the term $\log(R)$ can be avoided if $K$ has positive increase. A similar but easier calculation shows that $\norm{I_{35}}\leq C_1\Chat_0 e^{-t/M(R)}$, for a constant $C_1$ solely depending on $M(r_1),\ep\inv, r_1$ and $r_1\inv$. Therefore there exists $R_1\geq R_0$ only depending on $1/K,M(r_1),\ep\inv,r_1,r_1\inv$ and $\Chat_0/\Chat$ such that $\Chat\log(R)K(R)\geq C_1\Chat_0$ for all $R\geq R_1$. Thus $\norm{I_{35}}$ can be absorbed into the above estimate on $I_{36}$, if we assume $R\geq R_1$. Therefore, we assume $R\geq R_1$ in the following. 

Before we finally consider the integral $I_{34}$, let us first summarize what we obtained so far. By (\ref{eq: estimate on I1}), (\ref{eq: estimate on I2}), (\ref{eq: estimate on I37}) and (\ref{eq: estimate on I36}) we have for $R\geq R_1$
\begin{align}\label{eq: this is to optimize} 
 \norm{f(t)-I_{34}} \leq \frac{C}{R} \left( \norm{f'}_{\infty} + \Chat R\log(R)K(R)e^{-\frac{t}{M(R)}} \right) .
\end{align}
Here $C$ solely depends on $\ep\inv, C_\ep$ and $r_1\inv$. The choice $R=\MKtilde^{-1}(t)$ yields for the same value of $C$ that
\begin{align}\label{eq: estimate on f-I34}
 \normbig{f(t)-I_{34}} \leq \frac{C(\norm{f'}_{\infty}+\Chat)}{\MKtilde^{-1}(t)} 
\end{align}
for all $t\geq t_1 := \MKtilde(R_1)$, what we assume throughout the proof from now on.
If $K$ has positive increase one can improve (\ref{eq: estimate on I36}) by removing the term $\log(R)$ - as noted above. This allows us to remove the logarithmic term from (\ref{eq: this is to optimize}), which in turn allows us to define $\Ktilde=K_1$. 

Now let us turn to the estimation of $I_{34}$. Observe that $\ftilde$ satisfies for $0<x<r$
\begin{align*}
 \ftilde(x) = \frac{1}{2\pi i} \cdot \frac{1}{x} \cdot \lim_{\delta\downarrow0} (\ghat(-x+i\delta)-\ghat(-x-i\delta)).
\end{align*}
Note that, by assumptions on $f$, the terms to the right of the limit are uniformly bounded. Thus by dominated convergence and a change of variables (we replace $x$ by $-x$ in the parametrization of the first sub-path of $\gamma_{34}$) we get
\begin{align*}
 I_{34} 
 &= -\int_0^{\frac{1}{M(r_1)}} e^{-xt} \left[ 1 + (\psi - 1)\right](-R\inv x)  \left[ \ftilde_{n-1} + (\ftilde - \ftilde_{n-1}) \right](-x)  dx.
\end{align*}
We show now that neglecting the terms $\psi - 1$, $\ftilde - \ftilde_{n-1}$ and then integrating from $0$ to $\infty$ in the above integral produces an error of order at most $t^{-n-1} + 1/\MKtilde^{-1}(t)$. First, we observe, using boundedness of $\ftilde$, $\psi(0)=1$ and $\psi'(0)=0$, that there exists a $C>0$ depending on $1/M(r_1),\ep\inv$ such that for $t > 0$
\begin{align}\label{eq: improve I34 first}
 \normbig{\int_0^{\frac{1}{M(r_1)}} (\psi(-R\inv x) - 1) e^{-xt} \ftilde(-x) dx} \leq C \frac{\norm{\ftilde}_{(-r,0)}}{R^2} .
\end{align}
This is actually a rather crude estimate since we did not use the effect of the exponential function under the integral. However, it suffices for our purposes. Using the standard integral representation of the Gamma function and the standard remainder estimate for the Taylor polynomial we get
\begin{align}\label{eq: ftilde reminder estimate}
 \int_0^{\frac{1}{M(r_1)}} e^{-xt} \normbig{\ftilde(-x) - \ftilde_{n-1}(-x)} dx \leq \sup_{-r<x<0} \frac{\|\ftilde^{(n)}(x)\|}{t^{n+1}}.
\end{align}
The fact that $\ftilde_{n-1}$ is a polynomial, together with $\Ktilde(R)\geq R K(r_1)$ yields for all $t\geq t_1$, assuming without loss of generality that $t_1\geq 1$
\begin{align}\label{eq: improve I34 second}
 \normbig{ \int_{\frac{1}{M(r_1)}}^{\infty} e^{-xt} \ftilde_{n-1}(-x) dx } \leq C \Ktilde(R)^{-\frac{M(R)}{M(r_1)}} \norm{\ftilde}_n \leq C\frac{\norm{\ftilde}_n}{R}.
\end{align}
Here $C$ depends on $n,1/M(r_1),1/K(r_1)$. Let $a_j$ for $j\in\N$ be the $j$-th Taylor coefficient of $\ftilde$, that is $\ftilde(z)=\sum_{j=0}^{\infty} a_j z^j$. Then
\begin{align*}
 \frac{1}{t} \int_0^{\infty} e^{-y} \ftilde_{n-1}(-t^{-1} y) dy &= -\sum_{j=0}^{n-1} a_j (-t^{-1})^{j+1} \int_0^{\infty} e^{-y} y^j dy  \\
 &= -\sum_{j=0}^{n-1} a_j j! (-t^{-1})^{j+1} 
 = \ftilde_{k-1}\left( \frac{d}{dt} \right) \frac{1}{t}
\end{align*}
We proved the existence of a constant $C>0$ depending on $1/M(r_1),1/K(r_1),\ep\inv$ such that for all $t\geq t_0$
\begin{align}\label{eq: estimate on I34}
 \normbig{I_{34} + \ftilde_{n-1}\left( \frac{d}{dt} \right)t^{-1} } \leq \sup_{-r<x<0} \frac{\|\ftilde^{(n)}(x)\|}{t^{k+1}} + C\frac{\norm{\ftilde}_{(-r,0)}+\norm{\ftilde}_n}{\MKtilde^{-1}(t)}.
\end{align}
If we combine (\ref{eq: estimate on f-I34}) and (\ref{eq: estimate on I34}) we get the desired decay rate.

\textit{The case $m>1$.} In order to improve (\ref{eq: estimate on f-I34}) to 
\begin{align*}
 \normbig{f(t)-I_{34}} \leq \frac{C (\norm{f^{(m)}}_{\infty} + \Chat)}{\MKtildem^{-1}(t)^m} ,
\end{align*}
it suffices to improve (\ref{eq: this is to optimize}) to 
\begin{align*}
 \normbig{f(t)-I_{34}} \leq \frac{C}{R^m} \left( \norm{f^{(m)}}_{\infty} + \Chat R^{m}\log(R)K(R)e^{-\frac{t}{M(R)}} \right) .
\end{align*}
We can achieve this if we can replace the final $C/R$ bound in (\ref{eq: estimate on I1}), (\ref{eq: estimate on I2}) and (\ref{eq: estimate on I37}) by a bound $C/R^m$. For (\ref{eq: estimate on I37}) this is easy since the estimation above actually shows the better bound $C/R^{m'}$ for any $m'\in\N$. To get the better bound in (\ref{eq: estimate on I1}) we use that for example $\gamma_{11}(x)$ is bounded from below by a constant times $R$. Observing that
\begin{align*}
 \frac{1}{\gamma_{11}(x)^{m-1}} \left(-\frac{d}{ds}\right)^{m-1} e^{-(s-t)\gamma_{11}(x)} = e^{-(s-t)\gamma_{11}(x)}
\end{align*}
we see that an integration by parts argument - using that $g^{(m-1)}$ is bounded - yields the desired $C\norm{f^{(m)}}_{\infty}/R^m$ bound. Actually, performing the integration by parts yields boundary terms, involving $g(s),g'(s),\ldots, g^{(m-2)}(s)$ at $s=t$ and $s=0$. However, for $s=t$ exactly the same boundary terms with opposite sign occur if we do the same trick for the estimation of $I_{21}$. So if we estimate directly the sum $I_{11}+I_{21}$ and use that $\psi$ is symmetric we see that the boundary terms at $s=t$ cancel out. For $I_{12}+I_{22}$ we do the same trick and get the improved estimate for (\ref{eq: estimate on I1}) and (\ref{eq: estimate on I2}) if we manage to handle the boundary terms at $s=0$. Fortunately we get get rid of those boundary terms by assuming that all derivatives of $f$ up to order $m-1$ vanish at zero. To achieve this, we simply have to replace $f$ by $f-f_0$ with $f_0$ given by
\begin{align*}
 f_0(s) = \sum_{j=0}^{m-1} \frac{1}{j!} f^{(j)}(0)s^j \chi(s), s\geq 0,
\end{align*}
and $\chi\in C_c^\infty((-1,1);\C)$ being equal to $1$ in a neighbourhood of $0$.

The proof is complete if we can improve (\ref{eq: improve I34 first}) to
\begin{align*}
 \normbig{ \int_0^{\frac{1}{M(r_1)}} ( 1 - \psi(-R\inv x) ) e^{-xt} \ftilde(-x) dx} \leq C \frac{\norm{\ftilde}_{(-r,0)}}{R^m}
\end{align*}
and (\ref{eq: improve I34 second}) to 
\begin{align*}
 \normbig{ \int_{\frac{1}{M(r_1)}}^{\infty} e^{-xt} \ftilde_{n-1}(-x) dx } \leq C \frac{\norm{\ftilde}_n}{R^m} .
\end{align*}
The second goal is already satisfied and can be seen by the same argument as in case $m=1$, since $\Ktilde(R)\geq R^mK(r_1)$ in this case. We could achieve the first goal if $\psi^{(j)}(0)=0$ for all $j=1,\ldots,m-1$. Our current fudge factor does not satisfy this for $m > 2$. However, if we replace it by 
\begin{align*}
 \psi_m(z) = c_{mk} \exp \left( - \exp \left( \left(\frac{4m+2}{1+z^{4m+2}}\right)^k \right) \right), \text{ where } c_{mk}=e^{e^{(4m+2)^k}},
\end{align*}
then this property is satisfied. One only has to check now that this new fudge factor works as well as the old one in the other parts of the proof. In particular we mention that $\psi_m$ also satisfies Lemma \ref{lem:psi_near_poles}. The proof of Theorem \ref{thm:main} is finished.
\end{proof}

\begin{rem}
Unfortunately we are not able to prove Theorem \ref{thm:main} if we allow $\ep=0$ in the condition (\ref{eq:K_vs_M}) between $M$ and $K$.
However, we can slightly relax (\ref{eq:K_vs_M}) to the following constraint
\begin{align}\label{eq:K_vs_M_relax}
 K(s) = O\left(\exp\left(\exp\left( \frac{sM(s)}{\tilde{L}_{N,\varepsilon}(sM(s))} \right)\right)\right),\, s\to\infty
\end{align}
for some $\varepsilon>0$ and $N\in\N_1$. Here for $j, N\in\N_1$ and $s\geq0$ we denote
\begin{align*}
 \tilde{L}_{N,\varepsilon}(s) = L_1(s)\cdot \ldots \cdot L_{N-1}(s) \cdot L_N(s)^{1+\varepsilon}
 \text{ and }
 L_j(s) = \underbrace{\log \circ \ldots \circ \log}_{j \text{ times}} (1+j+s).
\end{align*}
The proof of Theorem \ref{thm:main} changes only in the choice of the fudge factor $\psi$ and the contours $\gamma_1,\gamma_2,\gamma_{31}$ and $\gamma_{35}$. What we need in the proof is that $\psi(R\cdot)$ is bounded in the domain enclosed by $\gamma_1+\gamma_2$, we have $\psi(z)=O(\absbig{\Re z}^{\infty})$ if $z\rightarrow\pm i$ within this domain and that we can control the absolute value of $\psi(R^{-1} z)K(R)$ for $\absbig{\Re z}\leq 1/(RM(R))$. See for example the estimation of $I_{37}$ in (\ref{eq: estimate on I37}) for the reason why we need the last mentioned control.

To achieve all these things we define (in case $m=1$) the fudge factor by
\begin{align*}
 \psi(z) = c_{nk} \exp \left( - \exp_{n+1} \left( \left( \frac{2}{1+z^2} \right)^k \right) \right)
\end{align*}
for a $k\in \N_1$ to be chosen. The positive real number $c_{nk}$ is chosen in such a way that $\psi(0)=1$. By $\exp_j$ we denote the composition of $j$ exponential functions. Moreover we define $\chi:\R_+\rightarrow\R$ by
\begin{align*}
 \chi(y) = \frac{y^{k+2}}{\prod_{j=1}^n \exp_j(y^{-k})}.
\end{align*}
Observe that $(\chi^{-1})'(x)=o(x^{-1})$ as $x\rightarrow 0$. In the definition of the contours we replace all occurrences of $x$ or $\absbig{x}$ by $\chi^{-1}(x)$ or $\chi^{-1}(\absbig{x})$. To get the desired control on the above mentioned product involving $\psi$ and $K$ it is crucial to generalize Lemma \ref{lem:psi_near_poles} in the following way. 
\begin{lem}
 Let $\varepsilon\in(0,1)$ and $k\in\N_1$ with $k > 2\varepsilon^{-1} - 2$. Let us denote $\tilde{L}_{N,\varepsilon}(s)=L_1(t)\cdot \ldots \cdot L_{N-1}(s) \cdot L_N(s)^{1+\varepsilon}$ for positive real numbers $s$. Then
 \begin{align*}
  \absbig{\psi(z)} \leq C \exp \left(-\exp \left( \frac{x^{-1}}{\tilde{L}_{N,\varepsilon}(x^{-1})} \right) \right)
 \end{align*}
 for all $z\in\C$ which can be represented as
 \begin{align*}
  z = x + i(1-y) \text{ or } z = x + i (-1+y),
 \end{align*}
 where $y\in(0,1)$ and  $\absbig{x} = \chi(y)$.
\end{lem}
\begin{proof}
 Without loss of generality $z=x+i(1-y)$. As in the proof of Lemma \ref{lem:psi_near_poles} we get 
 \begin{align*}
  \left( \frac{2}{1+z^2} \right)^k &= y^{-k} + ik y^{-k-1} O(\chi(y)) \\
 \end{align*}
 which in turn yields
 \begin{align*}
  \exp_N \left( \left( \frac{2}{1+z^2} \right)^k \right) 
              &= \exp_N(y^{-k}) + ik y^{-k-1} \left(\prod_{j=1}^n \exp_j(y^{-k})\right) O(\chi(y)) .
 \end{align*}
 This implies
 \begin{align*}
  \absbig{\psi(z)} \leq \exp(-(1+o(1))\exp_{N+1}(y^{-N})).
 \end{align*}
 Now let $\delta=2/k$. From the definition of $\chi$ it is not difficult to see that
 \begin{align*}
  \exp_N(y^{-k}) \geq \frac{x^{-1}}{\tilde{L}_{N,\delta}(x^{-1})}.
 \end{align*}
 This finishes the proof.
\end{proof}
\end{rem}

Just for the purpose of completeness we mention that our methods also show the following generalization of a main result of \cite{BaBoTo16}. For us it is not clear how to formulate a theorem on ``$L^p$-rates of decay'' allowing for a logarithmic singularity at zero. Therefore we restrict the next theorem to the case of no singularity at zero.

\begin{thm}\label{thm:main_Lp}
 Let $f:\R_+\rightarrow X$ be a locally integrable function such that its weak derivative $f^{(m)}$ of order $m\in \N_1$ is $p$-integrable for a $p\in(1,\infty)$. Let $M, K:\R_+\rightarrow (0,\infty)$ be continuous and non-decreasing functions satisfying
 \begin{equation}\label{eq:K_vs_M_Lp}
  \exists\ep\in(0,1): K(s) = O\left(e^{e^{(sM(s))^{1-\varepsilon}}}\right),\, s\to\infty.
 \end{equation}
 Assume furthermore that $\fhat$ extends analytically to $\Omega_M\cup\C_+$, continuously to $\cl{\Omega_M}\cup\C_+$ and
 \begin{align}\nonumber
  \norm{z\fhat(z)} \leq \Chat K(\absbig{\Im z}) \text{ for all } z\in \cl{\Omega_M} .
 \end{align}
 For any $\gamma>1+p\inv$ the mapping
 \begin{align}\nonumber
  t\mapsto \MKK{K_{m,\mathrm{log},\gamma}}^{-1}(t)^m \normbig{f(t)}, t\geq \MKK{\Ktilde_{m,\gamma}}(0)
 \end{align}
 is $p$-integrable. Here $K_{m,\mathrm{log},\gamma}(s)=s^m \log(e \vee s) (M(s)\vee K(s))^\gamma, s\geq 0$.
\end{thm}

\begin{proof}[Sketch of the proof]
 For simplicity we assume $m=1$ and we follow the pattern of the proof of Theorem \ref{thm:main}. We write shortly $\Ktilde$ for $K_{1,\mathrm{log},\gamma}$. Since there is no singularity at zero there is no need for a path $\gamma_{34}$ if we directly set $\delta=\lambda=0$. Our task is to estimate $I_1, I_2, I_{35}, \ldots ,I_{37}$ appropriately since the estimation of $I_{31}, I_{32}$ and $I_{33}$ follows the same pattern as the estimation of $I_{37}, I_{36}$ and $I_{35}$. With a foresight to the end of the proof we define $R=\MKK{\Ktilde}\inv(t)$. From (\ref{eq: estimate on I1}) and (\ref{eq: estimate on I2}) we see that 
 \begin{align*}
  \MKK{\Ktilde}\inv(t) \normbig{I_j(t)} \leq C P_{\MKK{\Ktilde}\inv(t)\inv} * \normbig{g} (t) 
  \text{ for } t\geq \MKK{\Ktilde}(0),\, j\in\{1,2\}.
 \end{align*}
 Here we emphasize the dependence of $I_j$ on $t$ since we treat it as a function depending on $t$ in the following. Thus the Carleson embedding theorem implies that the mapping $t\mapsto \MKK{\Ktilde}\inv(t) \normbig{I_j(t)}$ is $p$-integrable for $j\in\{1,2\}$. We refer to \cite[Section 4]{BaBoTo16} for more details on the Carleson measure argument which is involved here.
 
 Before we go on with the estimation of $I_{0}$ we first note that for a given $\beta\in(0,1)$ it is easy to see that  $t\mapsto t^\beta K(\MKK{\Ktilde}(t))$ is bounded from below uniformly for $t\geq \MKK{\Ktilde}(0)$ (Compare with \cite[Lemma 2.4]{St17}). The same argument which gives (\ref{eq: estimate on I36}) yields
 \begin{align*}
  \MKK{\Ktilde}\inv(t)\normbig{I_{36}(t)} \leq C R \log(R) K(R) e^{-\frac{t}{M(R)}} \leq [M\vee K](\MKK{\Ktilde}\inv(t))^{-(\gamma-1)}
 \end{align*}
 for all $t\geq t_1$, where $t_1>0$ has to be chosen sufficiently large. As in the proof of Theorem \ref{thm:main} one can absorb $\norm{I_{35}}$ by $\norm{I_{36}}$ at the cost of possibly increasing $t_0$. Since $\gamma>p\inv$ we can choose $\beta<1$ in such a way that $t\mapsto t^{-p\beta(\gamma-1)}$ is integrable on $(1,\infty)$. This shows that $t\mapsto \MKK{\Ktilde}\inv(t)\normbig{I_0(t)}$ is $p$-integrable.
 
 Finally we have to estimate $I_{37}$. From (\ref{eq: estimate on I37}) we deduce that for any $\alpha>0$ we get after choosing $k$ (see definition of $\psi$) sufficiently large (depending on $\alpha$) that
 \begin{align*}
  \MKK{\Ktilde}\inv(t) \normbig{I_{37}(t)} \leq C [M\vee K](\MKK{\Ktilde}\inv(t))^{-\alpha} \leq C t^{-\alpha\beta}.
 \end{align*}
 Chosing $\alpha$ large enough we see that also this expression is $p$-integrable, which finishes the proof.
\end{proof}

In the preprint \cite{St17} we proved a slightly weaker version of Theorem \ref{thm:main_Lp}. Instead of refining the proof presented in \cite{BaBoTo16} as done in the current article, in that article we refined the techniques of Chill and Seifert \cite{ChSe16}. Chill and Seifert essentially proved Theorem \ref{thm:main_Lp} in the special case $p=\infty$ with (\ref{eq:K_vs_M_Lp}) replaced by the stronger constraint $K(s)=O((1+s)M(s)),s\to\infty$.

In \cite{St17} we crucially use the Denjoy-Carleman theorem on quasi-analytic functions to relax this constraint to (\ref{eq:K_vs_M_Lp}). Our proof from that paper shows that we can even further relax the joint assumption on $M$ and $K$ to (\ref{eq:K_vs_M_relax}) but not to (\ref{eq:K_vs_M_Lp}) with $\ep=0$. So the rather different proofs from that paper and the current paper need (apparently) the same joint constraint on $M$ and $K$. This leads us to the question whether the condition (\ref{eq:K_vs_M_Lp}) is optimal in the sense that the theorem would be false if $1-\ep$ equal to or slightly larger than $1$ was allowed. We see some superficial connection to the Phragm\'en-Lindel\"of condition for strips, since it also involves a growth bound in terms of a function $s\mapsto e^{e^{cs}}$. We do not know if there is any deeper connection except the mere existence of a ``double exponential'' in both conditions.


\section{(Local) decay of \texorpdfstring{$C_0$}{C0}-semigroups}\label{sec:Loc_dec_C0-SGs}
The results of the preceding section can be applied to calculate local and also global decay rates for $C_0$-semigroups. To fix some of our notation, let $T=(T(t))_{t\geq0}$ be a $C_0$-semigroup on a Banach space $(X,\norm{\cdot})$ with generator $A:D(A)\rightarrow X$. We denote by
\begin{align}\nonumber
 \omega_0(T) = \inf \left\{ \omega\in\R; (t\mapsto \normbig{e^{-\omega t}T(t)}) \text{ is bounded on $\R_+$} \right\}
\end{align}
the so called \emph{growth bound} of $T$. Given $m\in\N_1, \omega>0$ and an $X$-valued analytic function $\Gtilde$, defined on some ball with radius $r\in(0,\omega)$ we define $\Gtilde_{m,\omega}(z)=(\omega-z)^{-m}\Gtilde(z)$ for $\absbig{z}<r$. By $\Gtilde_{m,\omega,j}$ we denote the Taylor polynomial of $\Gtilde_{m,\omega}$ of order $j$.

\begin{cor}[to Theorem \ref{thm:main}]\label{cor:main}
Let $T$ be a $C_0$-semigroup on a Banach space $X$ with generator $A$. Let $P_1, P_2$ be two bounded operators on $X$ and let $x\in X$. Assume that $(t\mapsto \normbig{P_2T(t)P_1x})$ is bounded. 
Let $M, K:\R_+\rightarrow (0,\infty)$ be continuous and non-decreasing functions such that for some $\ep\in(0,1)$ and $r_1,C_\ep>0$
\begin{equation}
 K(s) \leq C_\ep e^{e^{(sM(s))^{1-\varepsilon}}},\,  s\geq r_1.
\end{equation}
Let $G(z)=P_2(z-A)^{-1}P_1 x$ for $\Re z > \omega(T)$. We assume that $G$ extends analytically to $(\Omega_M\cup\C_+)\backslash\R_-$, continuously to $(\cl{\Omega_M}\cup\C_+)\backslash\R_-$ and that for some $\Chat>0$
\begin{align}\label{eq:C0_growth}
 \normbig{G(z)} \leq \Chat K(\absbig{\Im z}) \text{ for } z\in \Omega_M, \absbig{\Im z} > r_1 .
\end{align}
Assume furthermore that for some analytic function $\Gtilde:B_r\rightarrow X$ the mapping $z\mapsto G(z) - \Gtilde(z)\log(z)$ is analytic on $B_r$ for some $r\geq M(0)$. For all $m\in\N_1, \omega > \omega(T) \vee r$ there are $C_m, t_m>0$ such that for all $n\in\N_1$ and for all $t\geq t_m$
\begin{align}\label{eq:C0_rate}
 \normbig{P_2 T(t) (\omega-A)^{-m} P_1 x - \Gtilde_{m,\omega, n-1} \left( \frac{d}{dt} \right) t^{-1} } 
 \leq \sup_{-r<s<0} \frac{\normbig{ \Gtilde_{m,\omega}^{(n)}(s)} }{t^{n+1}} + \frac{C_m}{\MKtilde^{-1}(t)^m} .
\end{align}
Here, $\Ktilde=K_{m,\log}$. If in addition $K$ has positive increase of order $a\in(0,m-1]$, one can choose $\Ktilde=K_{m-a,\log}$, and if $K$ has positive increase of order $(m-1)+$, one can choose $\Ktilde=K_1$. 
\end{cor}

Our main interest in applying this theorem is to consider the case where $P_1$ and $P_2$ are not the identity. We think that a typical situation is that $M$ is a slowly increasing function (possibly constant) and $K$ is a (possibly much) faster increasing function. That is, we assume that the perturbed resolvent extends to a relatively large domain to the left of the imaginary axis, but may grow very quickly.

\begin{proof}[Proof of Corollary \ref{cor:main}]
 For $t\geq 0$ let us define $f(t)=P_2 T(t) (\omega-A)^{-m} P_1 x$. Then we have for $t>0$ and for $z\in\Omega_M$
 \begin{align}\label{eq:C0thm_fm}
  f^{(m)}(t) = P_2 T(t) [\omega(\omega-A)^{-1} - 1]^m P_1 x \text{ and } \\ \label{eq:C0thm_fhat}
  \fhat(z) = \sum_{j=0}^{m-1} (\omega-z)^{-(j+1)}P_2(\omega-A)^{-(m-j)}P_1 x + (\omega-z)^{-m}G(z) .
 \end{align}
 Using the fact that $(\omega-A)\inv$ is the (absolutely convergent) Laplace transform of $T$, it is not difficult to prove
 \begin{align*}
  \normbig{P_2T(t)(\omega-A)^{-j}P_1x} \leq \omega^{-j}  \sup_{s\geq t} \normbig{P_2T(s)P_1x},\, t\geq 0
 \end{align*}
 by induction on $j\in\N$. Together with (\ref{eq:C0thm_fm}) this shows that $f^{(m)}$ is bounded. Clearly, for $\absbig{z}\leq r$ we have $\ftilde(z)=(\omega-z)^{-m}\Gtilde(z)$ in the terminology of Theorem \ref{thm:main}. Moreover, by (\ref{eq:C0thm_fhat}) together with (\ref{eq:C0_growth}) for $z\in\Omega_M$ with $\absbig{\Im z}\geq r_1$ we have
 \begin{align}\label{eq:C0_crude_estimate}
  \normbig{\fhat(z)} \leq \frac{C}{\absbig{\omega-z}} + \absbig{\omega-z}^{-m} K(\absbig{\Im z}) \leq C\absbig{z}\inv K(\absbig{\Im z}) .
 \end{align}
 If $K$ has positive increase of index $a\in(0,m-1]$ we can improve (\ref{eq:C0_crude_estimate}) to
 \begin{align*}
  \normbig{\fhat(z)} \leq C\abs{z}^{-1-a} K(\absbig{\Im z}) 
 \end{align*}
 for all $z\in\Omega_M$ with large modulus. Note also, that if $K$ has positive increase of order $a+$, then $s\mapsto (1 + s)^{-a} K(s)$ has positive increase. Now the conclusion follows from Theorem \ref{thm:main}.
\end{proof}

\begin{rem}\label{rem:Ctilde_in_main_C0_local}
 Inspecting the proof and using Remark \ref{rem:Ctilde_in_main} we can find out the dependence of $t_m$ and $C_m$ on the hypotheses in Corollary \ref{cor:main_C0}. Therefore let us define
 \begin{align*}
  \Chat_0 &= \sup\{\norm{zG(z)}; z\in\Omega_M\backslash\R_-, \abs{\Im z}<r_1\}, \\
  \Ctilde &= \sup\{ \norm{P_2 T(t) P_1 x}; t\geq 0 \} ,\\
  C_A &= \sup\{ \norm{P_2 (\omega-A)^{-j} P_1 x}; j = 0, \ldots, m \}
 \end{align*}
 Observe that with the notation of the proof $\norm{f^{(m)}}\leq 2^m \Ctilde$. In the following $C'\geq 1$ (respectively $C'_0\geq 1$) describes a constant which can be seen as a non-decreasing function in the parameters $1/M(r_1)$, $C_A/\Chat K(r_1)$, $\omega, \omega\inv$ (respectively $1/M(r_1)$, $C_A/\Chat_0$, $\omega, \omega\inv$). Now $t_m$ can be chosen to be a non-decreasing function in $m,1/K,M(r_1),\ep\inv,r_1,r_1\inv,C'_0\Chat_0/C'\Chat$. With this choice of $t_m$ the constant $C_m$ can be chosen to be of the form
 \begin{align*}
  C_m=C'_m(\Ctilde+C'\Chat) + C''_m (\norm{\Gtilde_{m,\omega}}_{(-r,0)}+\norm{\Gtilde_{m,\omega}}_n)
 \end{align*}
  with $C'_m>0$ (respectively $C''_m$) being a non-decreasing function depending on $m,r_1\inv,\ep\inv,C_\ep$ from (\ref{eq:K_vs_M}) (respectively $m,n,\ep\inv,1/M(r_1),1/K(r_1)$). If the information about positive increase is used, $t_m$ also depends on $s_a$ and $C'_m$ also depends on $a\inv, C_a$ from (\ref{eq:pos_inc}).
\end{rem}

Let us specialize this result to the case $P_1=P_2=1$ and $T$ being a \emph{bounded} $C_0$-semigroup. Note that a logarithmic type singularity cannot occur in this setting - in other words $\Gtilde=0$. To simplify the presentation we restrict to the case $m=1$.

\begin{cor}[to Theorem \ref{thm:main}]\label{cor:main_C0_strip}
 Let $M,K:\R_+\to(0,\infty)$ be continuous non-decreasing functions such that there exist $\ep\in(0,1)$ and $C_\ep,r_1>0$ with
 \begin{equation}\nonumber
   K(s) \leq C_\ep e^{e^{(sM(s))^{1-\varepsilon}}} ,\,  s\geq r_1.
 \end{equation}
 Let $A$ be the generator of a bounded $C_0$-semigroup $T$, which satisfies $\Omega_M\subset\sigma(A)$ and $\normbig{(z-A)\inv}\leq \Chat K(\abs{\Im z}), z\in\Omega_M$ for a constant $\Chat>0$. There exist constants $C'_1,C',t_1>0$, solely depending on $M$ and $K$  such that
 \begin{equation}\nonumber
  \normbig{T(t)A\inv} \leq \frac{C'_1}{\MKtilde\inv(t)} \left( \sup_{\tau\geq0}\norm{T(\tau)} + C'\Chat \right)
 \end{equation}
 for all $t\geq t_1$. Here $\Ktilde=K_{1,\log}$, and if $K$ has positive increase one can also choose $\Ktilde=K_1$. Moreover, $t_1$ (respectively $C'$ and $C'_1$) can be chosen to be a non-decreasing function with respect to $1/K$, $K(r_1)$, $M(r_1)$, $1/M(r_1)$, $\ep\inv$, $r_1$, $r_1\inv$ (respectively $1/M(r_1)$, $\norm{A\inv}/\Chat K(r_1)$ and $r_1\inv,\ep\inv,C_\ep$). If information about positive increase of $K$ is used $t_1$ also depends on $s_a$ and $C'_1$ also depends on $a\inv, C_a$ from (\ref{eq:pos_inc}).
\end{cor}
\begin{proof}
  For a given $x\in X$ let us define $f_x(t)=T(t)A\inv x,t\geq0$. Then $\norm{f_x'}_{\infty}\leq \Ctilde\norm{x}$ and by the resolvent identity
 \begin{align*}
  \norm{z \fhat_x(z)} \leq \Chat \left( K(0) + K(\abs{\Im z}) \right) \norm{x} \leq 2\Chat K(\abs{\Im z}) \norm{x}
 \end{align*}
 for all $z\in\Omega_{M}$. The conclusion follows from Theorem \ref{thm:main} and Remark \ref{rem:Ctilde_in_main}. To see the dependence of $t_1$ on the hypothesis, note that, in the terminology of Remark \ref{rem:Ctilde_in_main}, one can choose $C'_0=\Chat K(r_1)\norm{x}$. Moreover, $C'\geq 1$ and $\norm{f}_1\leq \norm{A\inv}\norm{x}\leq \Chat K(r_1)\norm{x}$. This implies that the expression $C'_0\Chat_0/C'\Chat$, on which $t_1$ depends, can be estimated from above by a non-decreasing function, solely depending on $K(r_1)$ and $1/M(r_1)$.
\end{proof}

Sometimes it is easier to estimate the resolvent only along the imaginary axis, and not on a domain to the left of the imaginary axis. Below we formulate a version of Corollary \ref{cor:main_C0_strip} which takes this into account. Such a Theorem is possible since the resolvent identity extends the estimate on the imaginary axis to a certain domain to the left of the imaginary axis. However, if for some reason, one already knows a resolvent bound on a ``large'' domain to the left of the imaginary axis it is advisable to prefer Corollary \ref{cor:main_C0_strip} over Corollary \ref{cor:main_C0}. 

To see why, let us consider the bounded (even contractive) semigroup $T$ generated by $A-1$, where $A$ is as in \cite[Example 5.1.10]{ArBaHiNe2011}. This semigroup is not uniformly exponentially stable but $\norm{T(t)A\inv}\leq C e^{-t/2},t\geq0$ which is optimal with respect to the exponent $1/2$. Moreover, the spectrum of $A$ consists of the points $\{-1+i2^n; n\in\N\}$. It is not difficult to see that for each $\delta\in(0,1)$ the resolvent is bounded by $C(1+s)^{\alpha_\delta}$ on the strip $\Omega_{(1-\delta)\inv}$, where $\alpha_\delta=\log(\delta\inv)/\log(2)$. Taking $\delta\in(1/2,1)$ in an optimal way, Corollary \ref{cor:main_C0_strip} implies an exponential decay rate of for a certain exponent $c\in(1/4,1/2)$. That is, we recover the correct exponential decay rate up to a loss in the choice of $c$. On the other hand one can show that the optimal resolvent bound along the imaginary axis is $\norm{(is-A)\inv}\leq C\log(e\vee s),s\in\R$. Thus Corollary \ref{cor:main_C0} implies a decay rate of the form $Ce^{-\sqrt{ct}}$ -  which  is far from being optimal.

\begin{cor}[to Corollary \ref{cor:main_C0_strip}]\label{cor:main_C0}
 Let $M:\R_+\to(0,\infty)$ be a continuous non-decreasing function and $A$ the generator of a bounded $C_0$-semigroup $T$, which satisfies $i\R\subset\sigma(A)$ and $\normbig{(is-A)\inv}\leq M(\abs{s}), s\in\R$. For any $c\in(0,1)$ there exist constants $C'_1,C',t_1>0$, solely depending on $M$ such that
 \begin{equation}\label{eq:C0_decay}
  \normbig{T(t)A\inv} \leq \frac{C'_1}{\Mlog\inv(ct)} \left( \sup_{\tau\geq0}\norm{T(\tau)} + C'(1-c)\inv \right)
 \end{equation}
 for all $t\geq t_1$. Here $\Mlog(s)=M(s)\log(e \vee s M(s)), s\geq0$. The constant $t_1$ (respectively $C'$) can be chosen to be a non-decreasing function with respect to $1/M,M(1)$ (respectively $1/M(1)$). Moreover $C'_1$ can be chosen to be a purely numerical constant - i.e. not depending on $M$ or the semigroup at all.
\end{cor}
\begin{rem}
 We emphasize the \emph{linear} dependence of the decay rate with respect to the semigroup constant $\sup_{\tau\geq0}\norm{T(\tau)}$. With a scaling argument one can use this to deduce a generalization of Theorem \ref{cor:main_C0} to \emph{unbounded} $C_0$-semigroups. We refer to Rozendaal and Veraar \cite[Section 4.7]{RoVe17} for more details on this interesting argument.
\end{rem}
\begin{rem}\label{rem:c_t}
 Since $t_1$ does not depend on the value of $c$ one can choose $c=c_t$ to be a function in $t$. For some $M$ one can easily find an ``optimal'' rate of convergence of $c_t\to1,t\to\infty$ which even further improves the decay rate. We refer the reader to the end of this section where we show the effect of this argument for a simple example. 
\end{rem}
\begin{proof}[Proof of Corollary \ref{cor:main_C0}]
 Given $\theta\in(0,1)$, from $\normbig{(is-A)\inv}\leq M(\absbig{s}), s\in\R$ and the resolvent identity we may deduce that the resolvent extends analytically to the domain $\Omega_{\theta\inv M}$ and satisfies $\normbig{(z-A)\inv}\leq (1-\theta)\inv M(\absbig{\Im z}), z\in\cl{\Omega_{\theta\inv M}}$. By choosing $\theta\in(c,1)$ only slightly larger than $c$ the conclusion follows from Corollary \ref{cor:main_C0_strip} by taking (randomly) $\ep=1/2$, $r_1=1$ and $C_\ep=1$.
\end{proof}

Let us compare the quality of the decay rate obtained in Corollary \ref{cor:main_C0} with known results from the literature. To simplify our considerations we assume $\lim_{s\rightarrow\infty}M(s)$ $=\infty$ in the following. Our result is sharper than the result obtained in \cite{ChSe16}. There Chill and Seifert showed (\ref{eq:C0_decay}) only for $c\in(0,1/2)$. We are not aware of any other result in the Banach space setting determining admissible values of $c$ explicitly. So - in this generality, and with respect to the obtained decay rate - it seems that our result is the sharpest available in the literature.

If $T$ acts on a \emph{Hilbert space} we are in the range of applicability of \cite[Theorem 4.1]{RoSeSt17}. Without using any specific information about $M$ the cited theorem already implies 
\begin{align}\label{eq:decay_RoSeSt}
 \norm{T(t)A\inv}=O(\Mlog\inv(ct)\inv),t\to\infty
\end{align}
for any $c\in(0,1)$ and $\Mlog$ replaced by the function defined by $M(s)\log(e\vee s), s\geq0$. The reason for this is that any non-decreasing function has, in the terminology of \cite{RoSeSt17}, \emph{quasi-positive increase} with auxiliary function $s\mapsto\log(e\vee s),s\geq0$ and constant $c$ from \cite[equation (4.1)]{RoSeSt17} being equal to $e\inv$ (see also \cite[Remark 4.2 (b)]{RoSeSt17}). A typical situation where the concrete value of $c$ (in \ref{eq:C0_decay}) matters is when $M$ grows at a sub-polynomial rate. In this case the conclusion of Corollary \ref{cor:main_C0} coincides with the mentioned application of \cite[Theorem 4.1]{RoSeSt17}. On the other hand, for concrete functions $M$ one can find ``optimal'' auxiliary functions which improve the decay rate even further. For example, let $\alpha>0$ and $M(s)=\log(e\vee s)^{\alpha},s\geq 0$. Then the largest open interval of admissible values $c$ for which (\ref{eq:decay_RoSeSt}) is true is $(0,c_\alpha)$ with $c_\alpha=\alpha^{-\alpha}(1+\alpha)^{1+\alpha}$ \cite[Example 4.3 and 4.5]{RoSeSt17}. The consequences of this are two-fold. On the one hand this shows that in the Hilbert space setting, with respect to the rates, the results obtained in \cite{RoSeSt17} are superior to Corollary \ref{cor:main_C0}, and for certain choices of $M$ the rate obtained in \cite{RoSeSt17} is significantly faster (but note that the dependence of certain constants was not discussed in \cite{RoSeSt17}). On the other hand - in this generality - the constant $c$ in (\ref{eq:C0_decay}) can not be chosen larger than $1$ since $\lim_{\alpha\to0} c_\alpha = 1$. Thus, Corollary (\ref{cor:main_C0}) is (almost) optimal with respect to the admissible values of $c$ - up to the question whether $c=1$ would be admissible too.

There is another (rather artificial) example which shows that in general (\ref{eq:decay_RoSeSt}) is not true for $c>1$. Let $\delta>0$ and consider a normal operator $A$ on a Hilbert space with spectrum $\sigma(A)=\{-\delta+is; s\in\R\}$. This operator is the generator of a semigroup $T$ satisfying $\norm{T(t)A\inv}=\delta\inv e^{-\delta t}$. Actually the semigroup is \emph{uniformly} exponentially stable, but this is not important for our purpose. Since $\norm{(is-A)\inv}=\delta\inv, s\in\R$ we see that Corollary \ref{cor:main_C0} implies a decay rate of the form $e^{-c\delta t}$ for any $c\in(0,1)$. That is, we recover the original decay rate up to an arbitrary small exponential loss. Actually, using the idea of Remark \ref{rem:c_t} with $c_t=1-(\delta t)\inv$ we can reduce the loss to be of linear order. We further explore the optimality of Theorem \ref{thm:main} and its corollaries in Section \ref{sec:optimality}.


\section{Optimality of Theorem \ref{thm:main} and its corollaries}\label{sec:optimality}
In this section we show that Theorem \ref{thm:main} is essentially optimal in the sense that for many possible choices of $M$ and $K$ it is not possible to replace $\MKtilde$ in (\ref{eq:decay_rate}) by $c\inv\MKtilde$ for a $c\in(1,\infty)$ without invalidating the Theorem in general (see Theorem \ref{thm:optimality} below). To show this we use almost the same method as in \cite{BoTo10} but improve it on a technical level. There the authors showed the optimality of the decay rate (\ref{eq:decay_rate}) up to a constant in the very particular case that $M(s)=C(1\vee s^{\alpha})$ and $K(s)=C(1\vee s^\beta)$ for $\beta>\alpha/2>0$. To compare our result with Borichev's and Tomilov's result (\cite[Theorem 3.8]{BoTo10}) take into account Remark 3.10 from their paper. Also note that while we are considering decay to zero of a function $f$, Borichev and Timolov considered the decay to zero of $\fhat(0)-\int_0^t f(s)ds$. This explains the differences in the hypothesis of our results compared to their results but it is possible to translate our results to their setting and reversely. Theorem \ref{thm:optimality_C0} shows the optimality of Corollary \ref{cor:main_C0_strip} under the same assumptions on $M$ and $K$ as in Theorem \ref{thm:optimality}. To simplify the presentation we restrict it to the case $m=1$ and refer to Remark \ref{rem:m_larger_1} for the adjustments to be made if $m>1$.

Let $M,K:\R_+\to(0,\infty)$ be continuous non-decreasing functions. We frequently impose some of the following assumptions on such function.
\begin{enumerate}
 \item[(H1)] There exists a continuous non-decreasing function $N:\R_+\to[1,\infty)$ and a real number $s_0\geq0$ such that $M(s+s')\leq N(s')M(s)$ for all $s\geq s_0$, $s'\geq 0$.
 \item[(H1')] Same as (H1) but additionally requiring $N(0)=1$.
 \item[(H2)] The function $s\mapsto M(s)^{-1/2}K(s), s\geq0$ has positive increase.
\end{enumerate}
In the following, whenever we refer to a certain function called $N$, we mean the function occurring in (H1) and (H1'). Given $\alpha>0$ all functions which are defined by $\log(e\vee s)^\alpha, 1\vee s^\alpha, e^{\alpha s}$ satisfy (H1'). Moreover, the class of functions satisfying (H1) (respectively (H1')) is closed under multiplication and under taking positive real powers. Given $\alpha\in[0,\infty)$ the function $s\mapsto e^{s^{\alpha}}, s\geq0$ satisfies (H1') if and only if $\alpha\in[0,1]$.  If $M$ is bounded (H1) is satisfied for any $N$ given by a constant which is strictly larger than $1$. If $M$ is constant (H1') is satisfied with $N=1$.

\begin{prp}\label{prp:optimality}
 Let $\Ktilde,M:\R_+\rightarrow(0,\infty)$ be continuous and non-decreasing functions. Assume that $\Ktilde$ has positive increase and that $M$ satisfies (H1).
 For any real number $\gamma > N(0)$ there exist constants $c,C>0$ and a locally integrable function $f:\R_+\rightarrow\C$ with bounded weak derivative such that $\fhat$ extends analytically to $\Omega_M\cup\C_+$, extends continuously to $\cl{\Omega_M}\cup\C_+$, satisfies
 \begin{align}\label{eq:fhat_optprp}
  \abs{\fhat(z)} \leq C \frac{M(\absbig{\Im z})^\frac{1}{2}}{1+\absbig{\Im z}} \Ktilde(\absbig{\Im z})^{\gamma} \text{ for all } z\in\Omega_M
 \end{align}
 and fulfills
 \begin{align}\label{eq: optimality decay rate}
  \limsup_{t\rightarrow\infty} \MKtilde^{-1}(t)\absbig{f(t)} > c.
 \end{align}
 Moreover, one can choose $f$ in such a way that $\fhat$ extends to a strip to the left of the imaginary axis.
\end{prp}
\begin{rem}
 It might happen that the restriction of $\MKtilde$ to certain intervals of finite length is constant. This could even happen at an infinite number of intervals accumulating at infinity. If this happens we define $\MKtilde\inv$ to be the right-continuous right-inverse of $\MKtilde$. However, whenever we apply this proposition such a situation will not occur.
\end{rem}

Before we prove this proposition we need a lemma which is similar to \cite[Lemma 3.9]{BoTo10}. Given a compactly supported measure $\mu$ on $\C\backslash(\Omega_M\cup\C_+)$ we use the following notation for $z\in \Omega_M\cup\C_+$ and $t\geq0$
\begin{align*}
 \CT\mu(z) = \int \frac{1}{z-\zeta} d\mu(\zeta),\, 
 \LT\mu(t) = \int e^{t\zeta} d\mu(\zeta),\,
 \LT'\mu(t) = \int \zeta e^{t\zeta} d\mu(\zeta).
\end{align*}
Note that $\LT\mu$ and $\LT'\mu$ are uniformly continuous. Recall the notion of a \emph{purely numerical} constant from the end of the introduction.
\begin{lem}\label{lem:optimality}
 Let $M,\Ktilde$ and $\gamma$ be as in the hypothesis of Proposition \ref{prp:optimality}. There exist $\delta, t_0>0$, such that for all $\varepsilon>0$ there exists $k_0\in\N$ such that for any $k\in\N_{k_0}$ there exists a compactly supported Borel measure $\mu$ on $\C\backslash(\cl{\Omega_M}\cup\cl{\Omega_{2\delta\inv}}\cup\C_+)$ such that for all $z\in\Omega_M$ and $t\geq t_0$
 \begin{align}\label{eq: optimality 1}
  \absbig{\CT\mu(z)}  \leq \frac{C}{R} \delta^{-\frac{1}{2}} (k/\delta)^{\frac{1}{2}} \Ktilde(R)^\gamma 1_{[R-2\delta,R+2\delta]}(\Im z) + \varepsilon, \\ \label{eq:optinew}
  \absbig{\CT\mu(iR-1/M(R))}  \geq \frac{c}{R} \delta^{-\frac{1}{2}} (k/\delta)^{\frac{1}{2}} \Ktilde(R),
  \\ \label{eq: optimality 2}
  \absbig{\LT'\mu(t)} \leq C 1_{[\frac{k}{2\delta},\frac{2k}{\delta}]}(t) + \varepsilon, \\ \label{eq: optimality 3}
  \absbig{\LT\mu(t)}  \leq \frac{C}{R} 1_{[\frac{k}{2\delta},\frac{2k}{\delta}]}(t) + \frac{\varepsilon}{\max\{R,\MKtilde^{-1}(t)\}}, \\ \label{eq: optimality 4}
  \absbig{\LT\mu(k/\delta)} \geq \frac{c}{R} .
 \end{align}
 Here $R=\MKtilde\inv(k/\delta)$ and $c,C>0$ can be chosen to be purely numerical constants.
\end{lem}
\begin{rem}\label{rem:deltavsKtilde}
 In the proof of Proposition \ref{prp:optimality} we need an additional fact about the dependence of the possible choices for $\delta$ on $\Ktilde$. Let $\delta_0=4\limsup_{s\to\infty}\log(s)/\MKtilde(s)$ $\in[0,\infty)$. The value $\infty$ is excluded since $\Ktilde$ has positive increase. If $M$ is bounded one can choose any $\delta\geq \delta_1$, where $\delta_1>\delta_0$ is a constant which solely depends on $1/\lim_{s\to\infty}M(s)$ in a non-decreasing way. If $M$ is unbounded, one can choose any $\delta\leq \delta_2$, where $\delta_2>0$ ($=\delta_0$ in this case) is a constant solely depending on $N-N(0)$ in a non-increasing way. All this follows directly from the discussion at the end of Step 2 of the proof, where the value of $\delta$ is chosen.
\end{rem}
 \begin{proof}[Proof of Lemma \ref{lem:optimality}]
  Let $\delta,\delta_1>0$ and $k_0\geq 1$ be real numbers and a natural number to be fixed later. Throughout the proof we assume the natural number $k$ to be not smaller than $k_0$. By assuming $k_0$ to be large enough we may assume that $R=\MKtilde(k/\delta)>1\vee (s_0+2\delta)$ with $s_0$ as in (H1). Moreover, within this proof $c,C$ always denote purely numerical constants even if we do not say this explicitly at each occurrence (to avoid too many repetitions). Since $\Ktilde$ has positive increase, there exists $\alpha>0$ such that 
  \begin{align}\label{eq:MKtildeDEVlog}
   \liminf_{s\to\infty} \frac{\MKtilde(s)}{\log(s)} > \alpha .
  \end{align}
  In the following we impose the following constraint on $\delta$
  \begin{align}\label{eq:alphadelta}
   \alpha\delta > 4.
  \end{align}
  Note that in case $M$ is unbounded this is not really a constraint on $\delta$ since $\alpha$ can be chosen arbitrarily large. Only if $M$ is bounded this is a constraint on $\delta$ which says that we are not allowed to take $\delta$ too small. Let us define
  \begin{align*}
   w = iR - \delta,\, q = e^{2\pi i/(k+1)},\, \delta A = kl(k)
  \end{align*}
  where $l:\R_+\rightarrow(0,\infty)$ is given by $l(s)\geq 2\log(e\vee s)$. By $\delta_{z_0}$ we denote the Dirac-measure at $z_0\in\C$. Let us define
  \begin{align*}
   \mu = \frac{\tau}{R} \sum_{j=0}^k q^j\delta_{w+A^{-1}q^j}.
  \end{align*}
  The constant $\tau>0$ will be chosen later. Note that $\spt \mu \subset \C\backslash\Omega_{\delta_1\inv}$ for any $\delta_1<\delta$ provided $k_0$ is chosen large enough. Before we go on we state a simple lemma which will be frequently applied in the following without any reference to it (to avoid repetition). We omit the highly obvious proof.
  \begin{lem}
   Let $n>0$ be a real number. The function $s\mapsto s^n e^{-s},s\geq0$ has a unique maximum at $s=n$. For $s<n$ the function is strictly increasing and for $s>n$ it is strictly decreasing.
  \end{lem}
  
  \paragraph{\textbf{Part 1:} estimation of $\LT\mu$.} We distinguish the two cases $t\leq A$ and $t>A$.
  
  \emph{Case 1: $t\leq A$.} We calculate
  
  \begin{align*}
   \LT\mu(t) &= \frac{\tau}{R}\sum_{j=0}^k q^j e^{t(w+A^{-1}q^j)} \\
   &= \frac{\tau}{R} e^{tw} \sum_{m=0}^{\infty} \frac{1}{m!}\left( \frac{t}{A} \right)^m \sum_{j=0}^k q^{(m+1)j} \\
   &= \frac{\tau}{R} \cdot e^{tw} \frac{(k+1)t^k}{A^k k!} \cdot \sum_{n=1}^{\infty}\frac{k!}{(n(k+1)-1)!} \left( \frac{t}{A} \right)^{(n-1)(k+1)} \\
   &=: \frac{\tau}{R} \cdot I \cdot II.
  \end{align*}
  Clearly $II$ is bounded from below by $1$ and bounded from above by a constant which does not depend on $k$ or $A$. Thus by Stirling's formula we get
  \begin{align*}
   \absbig{\LT\mu(t)} \geq c\frac{\tau}{R} \sqrt{k} e^{-\delta t} \left( \frac{e \delta t}{\delta A k} \right)^k .
  \end{align*}
  Here, for $c$ one can choose any number from the interval $(0,1/\sqrt{2\pi})$, provided $k_0$ is chosen sufficiently large.
  As a function in $t$ we can maximize the right-hand side by setting $\delta t=k$. If we furthermore define
  \begin{align}\label{eq: definition of tau}
   \tau = \frac{1}{\sqrt{k}}(\delta A)^k
  \end{align}
  we see that (\ref{eq: optimality 4}) is proved. Since $II$ is bounded from above we have
  \begin{align}\label{eq: LT from above}
   \absbig{\LT\mu(t)} \leq C\frac{\tau}{R} \sqrt{k} e^{-\delta t} \left( \frac{e \delta t}{\delta A k} \right)^k .
  \end{align}
  Again we maximize the right-hand side by setting $\delta t = k$ and plugging in (\ref{eq: definition of tau}). This leads to
  \begin{align*}
   \absbig{\LT\mu(t)} \leq C\frac{\tau}{R} \sqrt{k} e^{-k} \left( \frac{e}{\delta A} \right)^k  = \frac{C}{R}
  \end{align*}
  For $t\in[k/2\delta,2k/\delta]$ this is already what we want to have in (\ref{eq: optimality 3}).

  \emph{Case 1.1: $\delta t \leq k/2$.} In this case the maximum in (\ref{eq: LT from above}) with respect to $t$ is attained for $\delta t=k/2$. This yields
  \begin{align*}
   \absbig{\LT\mu(t)} \leq C\frac{\tau}{R} \sqrt{k} e^{-\frac{k}{2}} \left( \frac{e}{2\delta A} \right)^k  
   = \frac{C}{R} \left( \frac{e}{4} \right)^{\frac{k}{2}} \leq \frac{\varepsilon}{R}
  \end{align*}
  For the last inequality we possibly have to increase $k_0$ depending on the smallness of $\ep$. We proved (\ref{eq: optimality 3}) for $\delta t \leq k/2$.

  \emph{Case 1.2: $2k \leq \delta t \leq \delta A$.} Note that (\ref{eq:MKtildeDEVlog}) yields $\MKtilde^{-1}(t)\leq e^{t/\alpha}, t\geq t_0$ as long as $t_0$ is large enough. Thus, if we multiply (\ref{eq: LT from above}) by $\MKtilde^{-1}(t)$ we get, after possibly increasing $k_0$ again
  \begin{align*}
   \MKtilde^{-1}(t)\absbig{\LT\mu(t)} &\leq C\frac{\tau}{R} \sqrt{k} e^{-(1-\frac{1}{\alpha\delta})\delta t} \left( \frac{e \delta t}{\delta A k} \right)^k \\
   &\leq \frac{C}{R} \left( \frac{2}{e^{1-\frac{2}{\alpha\delta}}} \right)^{k} \leq \ep .
  \end{align*}
  From the first to the second line we used that the maximum of the right-hand side of the first line is attained at $\delta t = 2k$ since $\alpha\delta\geq2$ (by (\ref{eq:alphadelta})). In the last estimate we used $e^{1-\frac{2}{\alpha\delta}}>2$. We proved (\ref{eq: optimality 3}) for $2k\leq \delta t\leq \delta A$.
  
  \emph{Case 2: $t>A$.} Then we have
  \begin{align*}
   \absbig{\LT\mu(t)} &\leq \frac{\tau}{R}(k+1)e^{-(\delta-A^{-1})t} \\
   &\leq \frac{C}{R} \sqrt{k} (\delta A)^k e^{-\delta A} e^{-(\delta-A^{-1})(t-A)} .
  \end{align*}
  Here, any $C\in(0,e\inv)$ can be chosen provided $k_0$ is large enough. In the following we assume that $\delta-A^{-1}>0$ which is true for large $k_0$. 
  
  \emph{Case 2.1: $A<t<2A$.} Using again $\MKtilde^{-1}(t)\leq e^{t/\alpha}$ for large $t$, we get
  \begin{align*}
   \MKtilde^{-1}(2A)\absbig{\LT\mu(t)} &\leq \frac{C}{R} \sqrt{k} \left( kl(k) e^{-l(k)} \right)^k e^{\frac{2kl(k)}{\alpha\delta}} \\
   &= \frac{C}{R} \sqrt{k} \left( kl(k) e^{-(1-\frac{2}{\alpha\delta})l(k)} \right)^k \leq \varepsilon
  \end{align*}
  since $(1-\frac{2}{\alpha\delta})>1/2$ (by (\ref{eq:alphadelta})) and $k_0$ is large enough. We proved (\ref{eq: optimality 3}) for $A<t<2A$.
 
   \emph{Case 2.2: $t\geq 2A$.} Using again (\ref{eq:alphadelta}), and in addition $\sqrt{k}(\delta A)^ke^{-\delta A}\leq 1$ for large $k_0$, we deduce
  \begin{align*}
   \MKtilde^{-1}(t)\absbig{\LT\mu(t)} &\leq \frac{C}{R} e^{-(1-\frac{1}{kl(k)})(\delta t-\delta A)} e^{\frac{t}{\alpha}} \\
   &\leq \frac{C}{R} e^{(\frac{1}{\alpha\delta}-\frac{1}{4})\delta t} \leq \varepsilon .
  \end{align*}
  This finishes the proof of (\ref{eq: optimality 3}).
 
  \paragraph{\textbf{Part 2:} estimation of $\CT\mu$.} First observe that as long as $z$ is no $(k+1)$-th root of unity we have
  \begin{align*}
   \sum_{j=0}^k \frac{q^j}{z-q^j} = \frac{k+1}{z^{k+1} - 1}.
  \end{align*}
  Clearly this equation must hold for some $k$-th order polynomial $p$ if one replaces the term $k+1$ on the right-hand side by $p(z)$. Moreover, the left-hand side is invariant under the substitution which replaces $z$ by $q z$. Thus $p(z)=p(qz)$. But this implies that $p$ is a constant. By plugging in $z=0$ we see that $p=k+1$.
 
  The observation yields for $z\in\Omega_M$
  \begin{align}\label{eq: nice formula for CTmu}
   \CT\mu(z) = \frac{\tau}{R} \frac{(k+1)A}{(A(z-w))^{k+1} - 1}.
  \end{align}
  Now it is not difficult to prove (\ref{eq: optimality 1}) for $\absbig{\Im z - R}>2\delta$. The latter condition implies $\absbig{z-w}>2\delta$. Thus, using (\ref{eq: nice formula for CTmu}) we get for $\absbig{\Im z - R}>2\delta$ and $k_0$ large:
  \begin{align*}
   \absbig{\CT\mu(z)} \leq C \frac{\tau}{R} kA (2\delta A)^{-k-1} \leq \frac{C\sqrt{k}}{2\delta R} 2^{-k} \leq \varepsilon .
  \end{align*}
  If we do not have $\absbig{\Im z - R}>2\delta$ we can merely estimate $\absbig{z-w}\geq \delta - 1/M(R-2\delta)$. This yields for $z\in\Omega_M$ with $\absbig{\Im z - R}\leq 2\delta$ and for all $\gamma_1>1$ and $\gamma>\gamma_1 N(2\delta)$
  \begin{align*}
   \absbig{\CT\mu(z)} &\leq C \frac{\tau}{R} kA \left(\delta A \left(1-\frac{1}{\delta M(R-2\delta)} \right) \right)^{-k-1} \\
   &\leq \frac{C\sqrt{k}}{\delta R} e^{\gamma_1\frac{\delta\inv k}{M(R-2\delta)}} \\
   &\leq \frac{C\sqrt{k}}{\delta R} e^{\gamma_1 N(2\delta)\frac{\delta\inv k}{M(R)}} \\
   &\leq \frac{C}{R} \delta^{-\frac{1}{2}} (k/\delta)^{\frac{1}{2}} \Ktilde(R)^\gamma.
  \end{align*}
  From the first to the second line we use the inequality $1-x\geq e^{-\gamma_1 x}$ which is valid for small $x\geq0$. If $M$ is bounded we may choose $\delta>4/\alpha$ (compare with (\ref{eq:alphadelta})) large enough to make use of this inequality. From the second to the third line we used (H1). In order to justify the step from the third to the fourth line we have to make sure that the difference $N(2\delta)-N(0)$ can be made as small as we like. We distinguish two cases: First we consider the case when $M$ is bounded and then the case when $M$ is unbounded. If $M$ is bounded, in (H1) we can take $N$ to be the constant function which is equal to $(1+\gamma)/2$. Thus $N(2\delta)-N(0)=0$ in this case. If $M$ is unbounded we are allowed to take $\delta$ as small as we wish in order to guarantee the smallness of $N(2\delta)-N(0)$ since $\alpha$ can now be chosen arbitrarily large in order to satisfy (\ref{eq:alphadelta}). This proves (\ref{eq: optimality 1}). Concerning (\ref{eq:optinew}) a reverse inequality for $z = iR-1/M(R)$ can be proved analogously but in an even simpler way by using the inequality $1-x\leq e^{-x}$ which is valid for all $x\geq0$.
  
  \paragraph{\textbf{Part 3:} estimation of $\LT'\mu$.} Finally we want to estimate the derivative of $\LT\mu$. 
  
  \emph{Case 1: $t\geq A$.} In this case we directly get for large $k_0$
  \begin{align*}
   \absbig{\LT'\mu(t)} &\leq \frac{\tau}{R} (k+1) (R+A^{-1}) e^{-(\delta-A^{-1})t} \\
   &\leq C \frac{\sqrt{k}}{R} (\delta A)^k R e^{-\delta A} \leq \varepsilon .
  \end{align*}
  
  \emph{Case 2: $t<A$.} Let us first get a different representation of $\LT'\mu$:
  \begin{align*}
   \LT'\mu(t) &= \frac{\tau}{R} \sum_{j=0}^k q^j(w+A^{-1}q^j) e^{(w+A^{-1}q^j)t} \\
   &= \frac{\tau}{R} e^{tw} \sum_{m=0}^{\infty} \frac{1}{m!} \left(\frac{t}{A}\right)^m \sum_{j=0}^k (wq^{(m+1)j}+A^{-1}q^{(m+2)j}) \\
   &= \frac{w}{R} \tau e^{tw} \frac{(k+1)t^k}{A^k k!} \sum_{n=1}^{\infty} \frac{k!}{(n(k+1)-1)!} \left( \frac{t}{A} \right)^{(n-1)(k+1)}
      \left[ 1 + \frac{n(k+1)-1}{wt} \right] .
  \end{align*}
  Note that if $t>t_0>0$, the series at the end of the calculation is bounded by a constant which only depends on the smallness of $t_0$. Without loss of generality we may assume that $t_0\geq 1$, which has the effect that the constant $C$ in the following does not depend on $t_0$, which in turn would not be allowed for purely numerical constants. Thus
  \begin{align}\nonumber
   \absbig{\LT'\mu(t)} &\leq C\tau \sqrt{k} e^{-\delta t} \left( \frac{e \delta t}{\delta A k} \right)^k \left[ 1 + \frac{k}{Rt}\right] \\ \label{eq: LT' estimate}
   &\leq C e^{-\delta t} \left( \frac{e \delta t}{k} \right)^k [1 + k] 
  \end{align}
  Note that (\ref{eq: LT' estimate}) as a function in $t$ assumes its maximum at $\delta t = k$. Therefore we see that $\absbig{\LT'\mu(t)}$ bounded by a purely numerical constant. This shows (\ref{eq: optimality 2}) for $k/2\delta \leq t \leq 2k/\delta$.

  \emph{Case 2.1: $\delta t \leq k/2$.} The maximum in (\ref{eq: LT' estimate}) is then attained for $\delta t = k/2$. This yields
  \begin{align*}
   \absbig{\LT'\mu(t)} &\leq C e^{-\frac{k}{2}} \left( \frac{e}{2} \right)^k \leq C \left( \frac{e}{4} \right)^{\frac{k}{2}} \leq \varepsilon
  \end{align*}
  if $k_0$ is large enough.
  
    \emph{Case 2.2: $2k\leq \delta t \leq A$.} The maximum in (\ref{eq: LT' estimate}) is then attained for $\delta t = 2k$. This yields
  \begin{align*}
   \absbig{\LT'\mu(t)} &\leq C e^{-2k} \left( 2e \right)^k \leq C \left( \frac{2}{e} \right)^{\frac{k}{2}} \leq \varepsilon
  \end{align*}
  if $k_0$ is large enough. This finishes the proof of Lemma \ref{lem:optimality}.
 \end{proof}

 \begin{proof}[Proof of Proposition \ref{prp:optimality}]
  For an $\varepsilon_0\in(0,1)$ to be chosen later we define a sequence $(\varepsilon_n)$ by $\varepsilon_n=2^{-n}\varepsilon_0$. There exists a $\delta>0$, an increasing sequence of natural numbers $(k_n)$ and a sequence of measures $(\mu_n)$ according to Lemma \ref{lem:optimality}. Actually we apply this lemma with $\Ktilde$ replaced by $\Ktilde'(s)=\Ktilde(0\vee(s-2\delta)),s\geq 0$ (see also Remark \ref{rem:deltavsKtilde}). The reason for this shift is, that now (\ref{eq: optimality 1}) implies
  \begin{align}\label{eq:optimality1_new}
     \abs{\CT\mu_n(z)}  \leq \frac{C}{\abs{z}} \MKtilde(\abs{\Im z})^{\frac{1}{2}} \Ktilde(\abs{\Im z})^\gamma 1_{[R_n-2\delta,R_n+2\delta]}(\Im z) + \varepsilon_n,
  \end{align}
  for some constant $C>0$ solely depending on $N$. Note that $\MKtilde\inv$ and $\MKK{\Ktilde'}\inv$ are asymptotically equivalent. We may assume that $([R_n-2\delta,R_n+2\delta])$ and $([k_n/2\delta,2k_n/\delta])$ are sequences of pairwise disjoint intervals. Let us define
  \begin{align*}
   f(t) = \sum_{n=1}^{\infty} \LT\mu_n(t) \text{ for } t\geq0.
  \end{align*}
  The series is uniformly convergent because of (\ref{eq: optimality 3}). Therefore, the function $f$ is continuous and since the sequence of derivatives converges uniformly on compact intervals (by (\ref{eq: optimality 2})) we see that $f$ has a bounded weak derivative given by
  \begin{align*}
   f'(t) = \sum_{n=1}^{\infty} \LT'\mu_n(t) \text{ for } t\geq0.
  \end{align*}
  By a similar argument the Laplace transform has the form
  \begin{align*}
   \fhat(z) = \sum_{n=1}^{\infty} \CT\mu_n(z)  \text{ for } z\in\Omega_M.
  \end{align*}
  Here the sum converges uniformly on compact subsets of $\Omega_M\cup\C_+$ (by (\ref{eq:optimality1_new})). We already know that the derivative of $f$ is bounded. The estimate (\ref{eq:fhat_optprp}) follows immediately from (\ref{eq:optimality1_new}), at the cost of possibly increasing $\gamma$ by an arbitrary small amount (because of an additional factor $\log(e\vee\Ktilde)^{1/2}$). It remains to prove (\ref{eq: optimality decay rate}). Let us set $t_n=k_n/\delta$. Then we deduce from (\ref{eq: optimality 3}) and (\ref{eq: optimality 4}) that
  \begin{align*}
   \absbig{f(t_n)} &\geq \frac{c}{R_n} - \varepsilon_0 \sum_{j\neq n} \frac{2^{-j}}{\max\{R_j, \MKtilde^{-1}(t_n)\}} \\
   &\geq \frac{c}{R_n} - \varepsilon_0 \sum_{j\neq n} \frac{2^{-j}}{R_n} \\
   &\geq \frac{c}{R_n} = \frac{c}{\MKtilde^{-1}(t_n)}.
  \end{align*}
  In the last line we chose $\varepsilon_0$ small enough.
 \end{proof}
 
\begin{rem}\label{rem:m_larger_1}
 By the same technique one can prove a generalization of Proposition \ref{prp:optimality} taking into account higher order derivatives of $f$. To achieve this one just has to define the measure $\mu$ from the proof of Lemma \ref{lem:optimality} by $\mu = \tau R^{-m} \sum_{j=0}^k q^j\delta_{w+A^{-1}q^j}$.
\end{rem}
 
Now we present a rather general condition on $M$ and $K$ for which Theorem \ref{thm:main} is essentially sharp. Therefore, given $\alpha>0,\beta>1/2$ let us abbreviate
\begin{align}\label{eq:abbr_cab}
 \cab =  \left( \frac{1}{2} +\sqrt{\frac{1}{4}+\frac{1}{\alpha}} \right) \vee \left( \frac{2\beta}{2\beta - 1} \right). 
\end{align}
Observe that $\cab\to1$ if $\alpha,\beta$ tend to infinity simultaneously.
 
\begin{thm}\label{thm:optimality}
 Let $K,M:\R_+\to(0,\infty)$ be continuous non-decreasing functions. Assume (H1'), (H2) and that for some $\alpha>0, \beta>1/2$
 \begin{align}\label{eq:nec_for_opt}
  \liminf_{s\to\infty} s^{-\alpha} M(s)^{-\beta} K(s) = \infty .
 \end{align}
 For any $c_1 > \cab$ (see (\ref{eq:abbr_cab})) there exist constants $c,C>0$ and a locally integrable function $f:\R_+\rightarrow\C$ with bounded weak derivative such that $\fhat$ extends analytically to $\Omega_M\cup\C_+$, extends continuously to $\cl{\Omega_M}\cup\C_+$, satisfies
 \begin{align}\label{eq:opt_fhat_cond}
  \abs{z\fhat(z)} \leq C K(\absbig{\Im z}) \text{ for all } z\in\Omega_M,
 \end{align}
 and fulfills
 \begin{align}\label{eq:opt_f_rate}
  \limsup_{t\rightarrow\infty} \MKK{K_1}^{-1}(c_1 t)\absbig{f(t)} \geq c.
 \end{align}
 Moreover, one can choose $f$ in such a way that $\fhat$ extends to a strip to the left of the imaginary axis. 
\end{thm}
Note that (\ref{eq:nec_for_opt}) excludes the case when $K$ compared to $M$ is ``relatively small''. Although in this article we are mainly interested in ``relatively large'' $K$, e.g. $M$ being constant while $K$ grows at a super-polynomial rate (this implies $\cab=1$), we think that it might be an interesting question to explore such a situation in future research. It is unknown if this leads to an improved decay rate - compared to the one obtained by our main result. For example, given $\alpha>0,\beta\in(0,\alpha/2]$ it is an open problem to decide whether there exists a function $f:\R_+\to\C$ satisfying the hypothesis of Theorem \ref{thm:main} (ignoring the part about the logarithmic singularity) with $M(s)=C(1\vee s^{\alpha})$ and $K(s)=C(1\vee s^\beta)$, which \emph{does not} decay like $t^{-1/\alpha}$.
\begin{proof}[Proof of Theorem \ref{thm:optimality}]
 Let $1<\gamma<c_1$ be a constant to be chosen later. Let $\Ktilde'':\R_+\to(0,\infty)$ be defined by $\Ktilde''(s)=\sup_{s'\leq s}(M(s')^{-1/2} K(s'))^{1/\gamma}$. Clearly $\Ktilde''$ is non-decreasing. Moreover, by (H2) this function has positive increase and there exists a constant $C$ such that for large $s$
 \begin{align*}
  (M(s)^{-1/2} K(s))^{1/\gamma} \leq \Ktilde''(s) \leq C (M(s)^{-1/2} K(s))^{1/\gamma} .
 \end{align*}
 Now we may apply Proposition \ref{prp:optimality} with $\Ktilde$ replaced by $\Ktilde''$ and find a function $f:\R_+\to\C$ as claimed in that proposition. Clearly, by the definition of $\Ktilde''$ the condition (\ref{eq:opt_fhat_cond}) is satisfied. Now let us fix $\gamma$ to be a constant between $1$ and $c_1/\cab$, e.g. $\gamma=(1+c_1/\cab)/2$. A short calculation shows that (\ref{eq:nec_for_opt}) implies
 \begin{align}\label{eq:K0vsK}
  \MKK{\Ktilde''}(s) \geq c_1\inv \MKK{K_1}(s) 
 \end{align}
 if $s$ is large enough. This yields (\ref{eq:opt_f_rate}).
\end{proof}

In the case of exponential decay one can see the optimality of Theorem \ref{thm:main} by a rather simple argument. In fact, let $f:\R_+\to X$ be any locally integrable function with $f'$ being bounded and $\norm{f(t)}\leq C_1 e^{-\delta t}, t\geq 0$ for some $\delta>0$. Then, necessarily $\fhat$ extends to the strip $\Omega_{(c\delta)\inv}$ for any $c\in(0,1)$ and is bounded by a constant which is proportional to $(1-c)\inv$. Therefore, Theorem \ref{thm:main} implies that $\norm{f(t)}\leq C_2 (1-c)\inv e^{-c\delta t/2},t\geq0$ for any $c\in(0,1)$ and a constant $C_2$ which does not depend on the choice of $c$. This means, that whenever exponential decay $e^{-\delta t}$ occurs and one has precise knowledge of the growth behaviour of the Laplace transform on any strip, by using Theorem \ref{thm:main}, one can always recover $\delta$ up to a correcting factor of at most $2+\ep$ for arbitrary small $\ep>0$. On the other hand, for certain \emph{concrete} examples, like $f(t)=e^{-\delta t}$ or orbits of the semigroup from \cite[Example 5.1.10]{ArBaHiNe2011} considered in the paragraphs between Corollary \ref{cor:main_C0_strip} and \ref{cor:main_C0}, the correcting factor \emph{can} be arbitrary close to $1$ or at least slightly smaller than $2$.

If one wants to determine the decay rate of a certain function with the help of our main result, one is not necessarily forced to take $M$ as small as possible. The reason is, that for relatively small $M$ one possibly has to choose $K$ so large that the obtained decay rate is worse than a rate obtained with the help of a larger choice of $M$ and thus also a possibly smaller choice of $K$. The above cited example from \cite{ArBaHiNe2011} illustrates this. In this example one could choose $M=M_\delta=\delta\inv$ for any $\delta\in(0,1)$ but the obtimal choice of $K_\delta$ is such that the optimal choice of $\delta$ (yielding the fastest decay estimate on the semigroup) is slightly larger than $1/2$. The aim of the following proposition and the examples afterwards is to further explore what can happen for different choices of $M$. To simplify the proof we impose a rather strong condition on $K$, compared to the other results in this section. One could formulate a more sophisticated version of the proposition, with a relaxed constraint. However, for our considerations in Example \ref{ex:take_M_large} we do not need a greater generality.

\begin{prp}\label{prp:Mtilde_worse}
 In the situation of Theorem \ref{thm:optimality}, assuming in addition that
 \begin{align}\label{eq:K_cond}
  \forall \gamma'>1, \delta\geq 0 \exists C>0 \forall s\geq 0: K(s+\delta) \leq C K(s)^{\gamma'}
 \end{align}
 one can choose $f$ in such a way that there exists a constant $c>0$ and a strictly increasing sequence of real numbers $R_n>0, n\in\N$ which tend to infinity such that for any continuous non-decreasing function $\Mtilde:\R_+\to(0,\infty)$ with $\Mtilde\geq M$ and for every $\theta\in[0,1]$
 \begin{align}\label{eq:Mtilde_worse}
  \abs{z_n \fhat(z_n)} \geq c \sqrt{\MKK{K_1}(R_n)} K_1(R_n)^{\frac{\theta M(R_n)}{c_1 \Mtilde(R_n)}},
 \end{align}
 where $z_n=iR_n+\theta/\Mtilde(R_n)$. Moreover, setting $t_n=\MKtilde(R_n),n\in\N$ we have that $\liminf_{n\to\infty} \MKtilde\inv(c_1t_n)\abs{f(t_n)}>0$.
\end{prp}
\begin{proof}
 We are in the setting of the proof of Theorem \ref{thm:optimality}. Let $\gamma$ and $\Ktilde''$ be as in that proof. Recall the construction of $f$ in the proof of Proposition \ref{prp:optimality} and let $\delta,\ep$ and $\mu_n,R_n,n\in\N$ be as in that proof. This time in the proof of Proposition \ref{prp:optimality}, while applying Lemma \ref{lem:optimality}, we do not replace $\Ktilde=\Ktilde''$ by $\Ktilde'$ since we can deduce (\ref{eq:optimality1_new}) by using (\ref{eq:K_cond}), at the cost of increasing $\gamma$ slightly. Note that the magnitude of $\fhat(z_n)$ is given by the magnitude of $\CT\mu_n(z_n)$ up to an error of order $2^{-n}\ep_0$. To estimate $\CT\mu_n(z_n)$ from below we consider equation (\ref{eq: nice formula for CTmu}) from Step 2 in the proof of Lemma \ref{lem:optimality}. Recall that $\delta\inv k_n = \MKK{\Ktilde''}(R_n)$. Together with the basic inequality $1-x\leq e^{-x}$, which is valid for all $x\geq0$ we deduce
 \begin{align*}
  \abs{\CT\mu_n(z_n)} &\geq \frac{\tau_n}{R_n} k_n A_n \left( \delta A_n\left( 1 - \frac{\delta\inv\theta}{\Mtilde(R_n)} \right) \right)^{-k_n-1} \\
  &\geq \frac{\sqrt{k_n}}{R_n}\delta\inv e^{\frac{\delta\inv k_n}{\Mtilde(R_n)}} \\
  &= \frac{1}{R_n} \delta^{-\frac{1}{2}} \sqrt{\MKK{\Ktilde''}(R_n)} \Ktilde''(R_n)^{\frac{\theta M(R_n)}{\Mtilde(R_n)}}
 \end{align*}
 Clearly, (\ref{eq:K0vsK}) is equivalent to $\Ktilde''(s)\geq K_1(s)^{1/c_1}$ for all $s\geq R_0$ if $R_0$ is chosen large enough. Thus, choosing $c,\ep_0>0$ sufficiently small yields the claim.
\end{proof}

Since $\MKtilde$ is unbounded, it is clear that for the particular $f$ from the above proposition we will never get an essentially faster decay rate in the conclusion of Theorem \ref{thm:main} if we replace $M$ by $\Mtilde$ and $K$ by the smallest continuous non-decreasing function which bounds the left-hand side of (\ref{eq:Mtilde_worse}) from above. This was of course already clear from Theorem \ref{thm:optimality} since $f$ was constructed to show the optimality of Theorem \ref{thm:main}. However, Proposition \ref{prp:Mtilde_worse} shows even more. If $\Mtilde$, is chosen sufficiently large, relative to $M$, the decay rate obtained by Theorem \ref{thm:main} gets significantly worse if we use $\Mtilde$ instead of $M$. We illustrate this behaviour by two specific examples.
\begin{ex}\label{ex:take_M_large}
 Let $f:\R_+\to\C$ be the function constructed in Proposition \ref{prp:Mtilde_worse}.

 (a) Assume that $M=1$ and $K(s)=e^{s^\alpha},s\geq0$ for some $\alpha>0$. Theorem \ref{thm:main} implies a decay rate of the form $t^{-1/\alpha}$, which is optimal by construction of $f$. Now let $\Mtilde(s)=1\vee s^{\alpha},s\geq0$. If we apply Theorem \ref{thm:main} with $M$ replaced by $\Mtilde$, because of the polynomial lower bound (\ref{eq:Mtilde_worse}) which is of the form $(1\vee s^{\alpha/2})\log(e\vee s)^{1/2}$, the best decay rate we can hope to deduce is of the form $(t/\log(e\vee t))^{-1/\alpha}$. That is, we get a logarithmic loss in that situation. Recall also, that there are functions as in Theorem \ref{thm:optimality} (or \cite[Theorem 3.8]{BoTo10}) corresponding to the choices $M=\Mtilde$ and $K(s)=1\vee s^{\alpha/2+\ep}$ (any $\ep>0$ allowed), which do not decay faster than $(t/\log(t))^{-1/\alpha}$. Actually, one can show with the help of the Phragm\'en-Lindel\"of theorem, that any function $f$ satisfying the hypothesis of Theorem \ref{thm:optimality} (ignoring the part with the logarithmic singularity) with $M=1$ and $K(s)=e^{s^\alpha},s\geq0$ is bounded by $\Ktilde(s)=C(1\vee s^{\alpha})$ on $\Omega_{\Mtilde}$ (see e.g. \cite[Proposition 1]{BoPe06}). However, by the above argumentation this is not sufficient to avoid the logarithmic loss. We do not know if one can improve the argument from \cite{BoPe06} to allow a $\Ktilde$ with $(1\vee s^{\alpha/2})\log(e\vee s)^{1/2}\lesssim\Ktilde(s)\lesssim 1\vee s^{\alpha/2+\ep}$ for any $\ep>0$. We do not think that this is possible.
 
 (b) Let $M=1$, $K(s)=1\vee s^\alpha$ for some $\alpha>0$. Theorem \ref{thm:main} yields a decay rate of the form $e^{-t/(1+\alpha)}$, which is optimal in the sense that for any $\ep>0$ one can choose $f$ in such a way that it does not decay faster than $e^{-t(1+\ep)/(1+\alpha)}$ if in addition $\alpha$ is large enough, depending on how small $\ep$ is. Let $\Mtilde(s)=\log(e\vee s),s\geq0$. By (\ref{eq:Mtilde_worse}) the best bound for $\abs{z\fhat(z)}$ on $\Omega_{\Mtilde}$, or on the imaginary axis, we can hope for is essentially given by $\log(e\vee \abs{\Im z})^{1/2}$. Therefore, with this choice of $\Mtilde$ we can use Theorem \ref{thm:main} to deduce a decay rate merely of the form $e^{-(ct)^{1/2}}$, for some $c>0$. This is a very dramatic loss compared to the actual decay rate! Again, this loss is essentially unavoidable. Indeed, by \cite[Example 4.5 (a)]{RoSeSt17} one can construct a normal semigroup on a Hilbert space, which is bounded by $2^{-1}\log(e\vee s)^{1/2}$ along the imaginary axis (and thus, up to a factor by the same bound also in $\Omega_{\Mtilde}$) but which decays (precisely) like $e^{-(ct)^{2/3}}$ for $c=3\sqrt{3}/2$.
\end{ex}

In the next theorem we prove the optimality of Corollary \ref{cor:main_C0_strip} in certain situations. If we compare (\ref{eq:opt_f_rate}) with (\ref{eq:optimality_C0_rate}) below we see that in Theorem \ref{thm:optimality_C0} we are able to replace the limes superior by a limes inferior. 

\begin{thm}\label{thm:optimality_C0}
 Let $K,M:\R_+\to(0,\infty)$ be continuous non-decreasing functions. Assume (H1'), (H2) and that for some $\alpha>0, \beta>1/2$
 \begin{align}\label{eq:nec_for_opt_C0}
  \liminf_{s\to\infty} s^{-\alpha} M(s)^{-\beta} K(s) = \infty .
 \end{align}
 For any $c_1 > \cab$ (see (\ref{eq:abbr_cab})) there exists a generator $A$ of a semigroup $T$ such that $\norm{(z-A)\inv}\leq \Chat K(\abs{\Im z}), z\in\Omega_M$ for some $\Chat>0$ and \begin{align}\label{eq:optimality_C0_rate}
  \liminf_{t\to\infty} \MKK{K_1}\inv(c_1t) \norm{T(t)A\inv}>0 .
 \end{align}
\end{thm}
\begin{proof}[Sketch of the proof]
 The first part of the proof is very close to the proof of \cite[Theorem 7.1]{BaBoTo16}. Therefore we only sketch it here and refer the reader to the paper of Batty, Borichev and Tomilov for the details. By $\buc$ let us denote the space of bounded and uniformly continuous functions on $\R_+$. Let us define the Banach space
 \begin{align*}
  X = \{f\in\buc; \exists C>0: \abs{\fhat(z)}\leq C K(\abs{\Im z}), z\in\Omega_M\}.
 \end{align*}
 We define the norm on this space to be $\norm{f}=\norm{f}_{\infty}+\inf C$, where the infimum ranges over all $C$ from the definition of the space $X$. Let us denote by $T$ the left shift semigroup on $X$. Following the lines of the proof of \cite[Theorem 7.1]{BaBoTo16} one can easily show that $T$ is bounded and that the resolvent satisfies the required estimate on $\Omega_M$. 

 Let $\ep=1$ and fix $\delta,t_0$, $k_0$ from Lemma \ref{lem:optimality}. We may assume that $k_0/\delta\geq t_0$. Let $\Ktilde'(s)=\sup_{s'\leq s}(M(0\vee (s'-2\delta))^{-1/2} K(0\vee(s'-2\delta)))^{1/\gamma}$ for a $\gamma\in(1,c_1/\cab)$. We want to use the lemma with $\Ktilde$ replaced by $\Ktilde'$. Note that the shift by $2\delta$ has no effect on $\MKK{\Ktilde'}\inv$ up to asymptotic equivalence. Recall that in the proof of Theorem \ref{thm:optimality} we used the lemma - indirectly via using Proposition \ref{prp:optimality} - with precisely this choice of $\Ktilde=\Ktilde'$. As in the proof of Theorem \ref{thm:optimality} we may use (\ref{eq:nec_for_opt_C0}) to show
 \begin{align}\label{eq:bla}
  \abs{\CT\mu_k(z)} \leq C(1+\abs{z})\inv K(\abs{\Im z}), z\in\Omega_M
 \end{align}
 for some $C$ not depending on $k$. Let $f_k=\LT\mu_k$. Clearly $f_k'=\LT'\mu_k$ and $\fhat_k=\CT\mu_k$. Since $f'_k$ is uniformly continuous (and bounded) $f_k,f_k'\in X$ and $A\inv f_k' = \int_0^\infty T(s) f_k' ds = f_k$. Here we also use that $\lim_{t\to\infty} f_k(t)=0$, by the Ingham-Karamata theorem since the Laplace transform of $f_k$ extends continuously to the imaginary axis (see e.g. \cite[Theorem 1.1]{ChSe16}). Actually, the same theorem implies that all functions from $X$ decay to zero at infinity. Note that by (\ref{eq:bla}), (\ref{eq: optimality 2}) and (\ref{eq: optimality 3}) the sequences $(f_k),(f'_k)$ are bounded in $X$. For $t_k=k/\delta$ with $k\geq k_0$ we deduce from (\ref{eq: optimality 4}) and (\ref{eq:K0vsK}) the existence of a constant $c>0$ not depending on $k$ such that
 \begin{align*}
  \norm{T(t_k)A\inv f_k'} \geq \norm{T(t_k)f_k}_{\infty} \geq \abs{f_k(t_k)} \geq \frac{c}{\MKK{K_1}\inv(c_1 t_k)} .
 \end{align*}
 Since $c_1>\cab$ was arbitrary, the same estimate also holds for $c_1$ replaced by $(1-\ep_1)c_1$, for a sufficiently small $\ep_1>0$ and for possibly different choices of $\delta,t_0,k_0$. Now (\ref{eq:optimality_C0_rate}) follows from the fact that $\{(1-\ep)t_k; k\geq k_0, \ep\in[0,\ep_1]\}$ contains an interval $[a,\infty)$ for a sufficiently large $a$.
\end{proof}

If (\ref{eq:nec_for_opt_C0}) is satisfied for \emph{all} $\alpha>0,\beta>1/2$ the above Theorem settles the question of optimality of Corollaries \ref{cor:main} and \ref{cor:main_C0_strip} almost entirely, since then $c_1>1$ is the only constraint on $c_1$. It would be desirable to prove the optimality of Corollary \ref{cor:main_C0} in the same spirit. More precisely, we ask whether there is a (reasonably large) class of (sub-polynomially growing) functions $M$ for which (\ref{eq:C0_decay}) is false if $c>1$ was allowed. Recall from the end of Section \ref{sec:Loc_dec_C0-SGs}, that in case of $M(s)=\log(e\vee s)^\alpha,s\geq 0$ with $\alpha>0$ there are $C_0$-semigroups for which (\ref{eq:C0_decay}) is false for any $c > \alpha^{-\alpha}(1+\alpha)^{1+\alpha} > 1$. These examples show that in general, $c>1$ is not allowed in (\ref{eq:C0_rate}). However, a positive answer to our question would be an even more striking result concerning optimality of Corollary \ref{cor:main_C0}. Unfortunately, Lemma \ref{lem:optimality} seems to be too weak to answer this question positively.
 

 \section{Application to a wave equation on an exterior domain}\label{sec:waves_ext_dom}

Let $\Omega\subsetneqq\R^d$ be a connected open set with bounded complement and non-empty $C^{\infty}$-boundary. The dimension $d$ is assumed to be at least $2$. We consider the wave equation on this domain:
\begin{align}\label{eq: exterior wave equation}
 \left\{
 \begin{array}{lr}
  u_{tt}(t,x) - \Delta u(t,x) = 0 &  (t\in(0,\infty), x\in\Omega), \\
  u(t,x) = 0 & (t\in(0,\infty), x\in\partial\Omega), \\
  u(0,x) = u_0(x), u_t(0,x) = u_1(x) & (x\in\Omega) .
  \end{array}
 \right.
\end{align}
Let us fix a radius $\rho>0$ such that the obstacle $\obstacle=\R^d\backslash\Omega$ is included in the open ball $B_{\rho}$ of radius $\rho$ and center $0$. We define a \emph{state} (at time $t$) of the system by $\xnice(t):=(u,v)(t):=(u(t),u_t(t))$. We define the local energy of a state by
\begin{align}\label{eq: local energy}
 \Eloc(\xnice) = \int_{\Omega\cap B_\rho} \absbig{\nabla u}^2 + \absbig{v}^2 dx .
\end{align}
Clearly, equation (\ref{eq: local energy}) is well defined for all $u\in C_c^{\infty}(\Omega)$ and $v\in L^2(\Omega)$. Therefore, it is also well defined on the \emph{energy space}
\begin{align}\nonumber
 \Ho = H_D^1(\Omega) \times L^2(\Omega),
\end{align}
where $H_D^1(\Omega)$ is the completion of $C_c^{\infty}(\Omega)$ with respect to the norm $(\int_{\Omega}\abs{\nabla u}^2)^{1/2}$.  We remark at this point that for any compactly supported $C^{\infty}$-function $\chi:\R^d\rightarrow\C$ the corresponding multiplication operator $f\mapsto \chi f$ is continuous from $H^1_D(\Omega)$ to $H^1_D(\Omega)$ and $L^2(\Omega)$. This is not completely obvious since $H^1_D(\Omega)$ is not a subspace of $L^2(\Omega)$ and actually the statement would be false if $\partial\Omega=\emptyset$. Fortunately we have assumed $\partial\Omega\neq\emptyset,\partial\Omega\in C^\infty$ and therefore the statement follows from the Poincar\'{e}-Steklov inequality applied to the open set $\Omega\cap B_r$ where the radius $r>0$ is chosen so large that $\Omega\cap B_r\neq\emptyset$ is connected and the support of $\chi$ is contained in $B_r$.

Let $m\in\N_0$. We are interested in the uniform decay rate of the local energy with respect to sufficiently smooth initial data, compactly supported in the ball of radius $\rho$:
\begin{align}\label{eq: pm}
 p_m(t) := \sup \left\{ \left( \frac{\Eloc(\xnice(t))}{\normbig{\xnice_0}^2_{H^{m+1}\times H^m}} \right)^{\frac{1}{2}} ; \xnice_0\in \Hcomp{m+1}\times \Hcomp{m}(\Omega\cap B_{\rho}) \right\}.
\end{align}
Here, by $\Hcomp{m}(\Omega\cap B_{\rho})$ we denote all square-integrable functions, \emph{compactly} supported on $\Omega\cap B_{\rho}$ for which all weak derivatives up to order $m$ are square-integrable too. We also write $\Ltwocomp=\Hcomp{0}$. It is well known that $p_0$ either does not decay to zero, or decays exponentially for $d$ odd and like $t^{-d}$ for $d$ even. Moreover, the decay can be characterized by boundedness of the local resolvent of $\A$ on the imaginary axis. We refer to \cite{Vo99} and references therein for these facts.

\subsection{The associated unitary \texorpdfstring{$C_0$}{C0}-group, its generator and basic properties of the truncated outgoing resolvent}
The wave equation (\ref{eq: exterior wave equation}) on the energy space $\Ho$ can be reformulated in the language of $C_0$-semigroups. Therefore, as above, we set $\xnice(t)=(u(t),u_t(t))$, $\xnice_0=(u_0,u_1)$ and write
\begin{align}\label{eq: exterior CP}
  \left\{
  \begin{aligned}
  \dot{\xnice}(t) = \A \xnice(t) , \\
  \xnice(0) = \xnice_0 \in \Ho  ,
 \end{aligned}
 \right.
\end{align}
where
\begin{align*}
  \A = \left(
 \begin{array}{cc}
  0       & 1 \\
  \Delta  & 0     \\
 \end{array}\right)
 \text{ with }
 D(\A) = D_\Delta \times (\underbrace{H^1_D\cap L^2}_{H^1_0})(\Omega) .
\end{align*}
Here $D_\Delta=\{u\in H^1_D(\Omega); \Delta u\in L^2(\Omega)\}$, where $\Delta$ denotes the Laplace operator in the sense of distributions. It can be proved that the wave operator $\A$ is skew-adjoint (see e.g. \cite[Theorem V.1.2]{LaPh1967}). Therefore the following theorem follows by Stone's theorem (see e.g. \cite[Appendix 1, Theorem 2]{LaPh1967}).
\begin{thm}
 The operator $\A$ generates a unitary $C_0$-group on $\Ho$.
\end{thm}

In the following we investigate the resolvent of $\A$ to get decay rates $p_m$ for the local energy. In the literature on local energy decay it is common to investigate the \emph{outgoing resolvent} of the stationary wave equation. For $\Re z>0$ and $f\in L^2(\Omega)$ the outgoing resolvent is defined as the Laplace transform
\begin{align*}
 R(z)f = \int_0^{\infty} e^{-zt} u(t) dt
\end{align*}
where $u$ is the first component of the solution to (\ref{eq: exterior CP}) for $\xnice_0=(0,f)\in\Ho$. Taking the Laplace transform of (\ref{eq: exterior CP}) it is not difficult to show that $w=R(z)f$ for $\Re z>0$ and $f\in L^2(\Omega)$ is the unique distributional solution in $L^2(\Omega)$ to the stationary wave equation
\begin{align}\label{eq: stationary exterior wave equation}
 \left\{
 \begin{array}{lr}
  z^2 w(x) - \Delta w(x) = f(x) &  (x\in\Omega), \\
  w(x) = 0 & (x\in\partial\Omega). \\
  \end{array}
 \right.
\end{align}
That is, $R(z)=(z^2-\Delta_0)^{-1}$ where by $\Delta_0$ we denote the Dirichlet-Laplace operator with domain $D(\Delta_0)=\{u\in H^1_0(\Omega); \Delta u\in L^2(\Omega)\}$. We emphasize that $D(\Delta_0)\neq D_\Delta$. There is an important relation between $R$ and the resolvent of $\A$: For $\Re z > 0$ we have
\begin{align}\label{eq: resolvents R vs A}
    (z-\A)^{-1} = \left(
 \begin{array}{cc}
  zR(z)       & R(z)   \\
  z^2R(z)-1   & zR(z)  \\
 \end{array}\right) .
\end{align}
Let us fix a cut-off function $\chi\in C_c^{\infty}(\R^d)$ with $0\leq \chi \leq 1$ such that $\chi=1$ on a neighbourhood of $\obstacle$. We define the truncated resolvent by $R_{\chi}(z)=\chi R(z) \chi$, where we consider $\chi$ as a multiplication operator on $L^2(\Omega)$. From the definition we see that the outgoing truncated resolvent is an analytic function on $\C_+$. The next proposition illuminates its behaviour on the other half of the complex plane.
\begin{prp}\label{prp:Burq_LaxPhillips_Vodev}
 (i)\text{\cite[Appendix B]{Bu98}} The truncated outgoing resolvent $R_{\chi}$ extends analytically to a neighbourhood of $i\R\backslash\{0\}$. Moreover, for any open sector $S\supseteq \R_-$ with vertex at $0$ the operator $R_{\chi}(z):L^2(\Omega)\rightarrow L^2(\Omega)$ is uniformly bounded for $z$ in a small neighbourhood of $0$ outside the sector $S$. (ii)\cite[Corollary V.3.3 together with Remark V.4.3]{LaPh1967} If the dimension $d\geq 3$ is odd, $R_{\chi}$ extends meromorphically to $\C$. (iii)\cite[Proposition 3.1]{Vo99} If the dimension $d\geq 2$ is even, then $R_{\chi}$ extends meromorphically to $\C\backslash\R_-$ and there exists a rank one operator $R_0$ such that
 \begin{align*}
  z\mapsto R_{\chi}(z) - R_0 z^{d-2}\log(z) \text{ is analytic}
 \end{align*}
 in a neighbourhood of $0$.
\end{prp}
Since the spectrum of $\Delta_0$ is $(-\infty,0]$ the (maximal) domain of analyticity of the operator-valued function $R$ is the interior of $\C_+$. In particular, $R$ does not extend across the imaginary axis if we consider it as a $\Lin(L^2(\Omega))$-valued function. However, if we consider $R(z)$ as an operator $R(z):\Ltwocomp(\overline{\Omega})\rightarrow \Ltwoloc(\Omega)$, then the above proposition says that $R$, with these values, does extend across the imaginary axis. Moreover, if $f\in \Ltwocomp(\overline{\Omega})$ and $z\in\C$ is such that $R(z)$ is defined, the function $w=R(z)f\in \Ltwoloc(\Omega)$ is a solution to (\ref{eq: stationary exterior wave equation}). For $\Re z<0$ the function $w$ thus defined is not necessarily in $L^2(\Omega)$ and in particular it need not be the unique $L^2$-solution of (\ref{eq: stationary exterior wave equation}). In other words, $R_{\chi}(z)\neq \chi (z^2-\Delta_0)^{-1}\chi$ if $\Re z<0$.

Let us define the analytic function $G_{\chi}: \C_+ \to \Lin(\Ho)$ by
\begin{align*}
 G_{\chi}(z) = \chi(z-\A)^{-1}\chi .
\end{align*}
Here, we consider $\chi$ as an operator on $\Ho$ acting as $\chi(u_0,u_1)=(\chi u_0, \chi u_1)$. In case $d\geq 3$ is odd, by Proposition \ref{prp:Burq_LaxPhillips_Vodev} together with (\ref{eq: resolvents R vs A}), we immediately see that $G$ extends to a meromorphic function on $\C$ which has no poles on $i\R$. If $d\geq 2$ is even, then $G_{\chi}$ extends to a meromorphic function on $\C_+\backslash\R_-$. Moreover, by Proposition \ref{prp:Burq_LaxPhillips_Vodev}(iii) together with (\ref{eq: resolvents R vs A}) (see also \cite[Remark 3.2]{Vo99}) there exists a finite rank operator $P_0$ such that
\begin{align}\label{eq:G_has_log-sing}
 z \mapsto G_{\chi}(z) - P_0 z^{d-1}\log(z) \text{ is analytic}
\end{align}
in a small ball around $0$. Since the spectrum of $\A$ is the entire imaginary axis (this follows from $\sigma(\Delta_0)=(-\infty,0]$) the equality $G_{\chi}(z)= \chi(z-\A)^{-1}\chi$ \emph{does not hold} for $\Re z<0$ in general. 
  
The following proposition seems to be well-known. Unfortunately we could not find a complete proof in the literature. Therefore we give a proof in the Appendix.  
\begin{prp}\label{prp:resolvent_A_vs_R}
 Let $\delta>0$ and let $\tilde{\chi}$ be defined as $\chi$ but with $\tilde{\chi}=1$ on a neighbourhood of the support of $\chi$. Let $z$ with $-\delta<\Re z < 0$ be no pole of $R_{\chi}$, then
 \begin{align*}
  \normbig{G_\chi(z)} \leq C \left( (1\vee\absbig{z})^{-1} + \absbig{z} \normbig{R_{\tilde{\chi}}(z)}_{L^2\rightarrow L^2} \right)
 \end{align*}
 with a constant $C>0$ independent of $z$. The reverse inequality - with a different constant, ignoring the first summand on the right hand side and $\tilde{\chi}$ replaced by $\chi$ - is also true.
\end{prp}

\subsection{Decay of the local energy}\label{sec: Decay of the local energy}
  
It can happen that a whole strip $\{ z\in\C ; -\delta < \Re z < 0 \}$ is free of poles of $G_{\chi}$ - see for instance \cite{Ik88}. In \cite{BoPe06} Bony and Petkov studied the impact of the presence of such a strip on local energy decay. There it was shown in a first step that such a strip implies that the norm of $G_{\chi}$ can be estimated by $C\exp(C\absbig{\Im(z)}^{\alpha})$ for large $z$ on this strip, and for some $\alpha>0$. Indeed $\alpha=d-1$ in this article but it was not shown that this is optimal. In a second step the authors showed that this implies a bound of the form $(1+\absbig{\Im z})^{\alpha}$ on $G_{\chi}$ for large arguments in a region of the form $\{ z\in\C; -c(1+\absbig{\Im z})^{-\alpha} < \Re z < 0 \}$. This step is rather abstract and relies only on the fact that $G_\chi$ is an analytic function on $\Omega_M\cup\C_+$, which is bounded from above by $C/\Re z$ on $\C_+$. Note that by Example \ref{ex:take_M_large} (a) the obtained bound on that region is (at least almost) the best bound one can hope for - under such general assumptions. Finally, in a third step they applied a Tauberian theorem (more precisely, \cite[Proposition 1.4]{PoVo99}) to get, for $d$ odd, a $(\log(t)/t)^{m/\alpha}$ decay rate for the local energy. If $d$ is even one gets a $t^{-d}\vee(\log(t)/t)^{m/\alpha}$ decay rate.
  
In the following we get rid of the logarithmic term, and simplify the proof compared to \cite{BoPe06}, by using a single application of Corollary \ref{cor:main} to the local resolvent on a strip. To present a more general result we consider the following conditions.
\begin{enumerate}
 \item[(a)] There is a continuous and non-decreasing function $M:\R_+\rightarrow(0,\infty)$ such that $R_{\chi}$ has no poles in $\Omega_M$.
 \item[(b)] $R_\chi$ extends analytically to $(\Omega_M\cup\C_+)\backslash\R_-$ and continuously to $(\cl{\Omega_M}\cup\C_+)\backslash\R_-$. Moreover, there is a real number $r_1>0$ and a continuous and non-decreasing function $K:\R_+\rightarrow(0,\infty)$ such that 
 \begin{align*}
  \absbig{\Im z} \normbig{R_{\chi}(z)}_{L^2\rightarrow L^2} \leq C K(\abs{\Im z})
 \end{align*}
 for all $z\in \Omega_M$ with $\absbig{\Im z} > r_1$.
 \item[(c)] There exists $a\in[0,\infty)$ such that $K$ has positive increase of order $a$.
\end{enumerate}
 Observe that $(c)$ is always satisfied for $a=0$ (by definition of positive increase of order $0$).
\begin{thm}\label{thm:application}
 Let $m\in\N_1$ and assume that the conditions (a-c) above are satisfied. (i) If $d\geq3$ is odd, there is a $C>0$ such that
 \begin{align}\nonumber
  p_m(t) \leq \frac{C}{\MKtilde^{-1}(t)^m} \text{ for every } t\geq M(0) .
 \end{align}
 (ii) If $d\geq2$ is even, there is a $C>0$ such that
 \begin{align}\nonumber
  p_m(t) \leq C \max \left\{\frac{1}{t^d}, \frac{1}{\MKtilde^{-1}(t)^m} \right\} \text{ for every } t\geq M(0) .
 \end{align}
 In both cases $\Ktilde=K_{(m-a)\vee1,\log}$. If $a>m-1$, we set $\Ktilde=K_1$.
\end{thm}
\begin{proof}
 (i) For $\Re z > 0$, let $G_{\chi}(z)=\chi (z-\A)^{-1} \chi$. Assumptions (a) and (b) together with Proposition \ref{prp:resolvent_A_vs_R} imply that $G_{\chi}$ extends analytically to $\Omega_M\cup \C_+$, and satisfies
 \begin{align}\nonumber
  \normbig{G_{\chi}(z)} \leq C K(\absbig{\Im z})  \text{ for } z\in \cl{\Omega_M} .
 \end{align}
 Thus, by Corollary \ref{cor:main}, for every $\xnice_0\in\Ho$ 
 \begin{align}
  \normbig{\chi e^{t\A} (1-\A)^{-m} \chi \xnice_0} \leq \frac{C}{\MKtilde^{-1}(t)^m} \normbig{\xnice_0} .
 \end{align}
 By the closed graph theorem the constant $C$ does not depend on $\xnice_0$. For simplicity we assume $m=1$ in the following. The general case can be treated in almost the same way.
   
 Let $\chi_1\in C_c^{\infty}(\R^d)$ be a function such that $0\leq \chi_1 \leq 1$ and $\chi_1=1$ on $\spt \chi$. Of course, Propositions \ref{prp:Burq_LaxPhillips_Vodev} and \ref{prp:resolvent_A_vs_R} remain valid if one replaces $\chi$ by $\chi_1$. Note that the commutator $[\chi,1-\A]$ is a bounded operator on $\Ho$. Let $\xnice_1=(1-\A)^{-1}\xnice_0\in D(\A)$. By Corollary \ref{cor:main},
 \begin{align*}
  \normbig{\chi e^{t\A} \chi \xnice_1} 
  &\leq \normbig{\chi e^{t\A} (1-\A)^{-1} \chi \xnice_0} + \normbig{\chi (\chi_1 e^{t\A} (1-\A)^{-1} \chi_1) [\chi,(1-\A)] \xnice_1} \\
  &\leq \frac{C}{\MKtilde^{-1}(t)} (\normbig{\xnice_0} + \normbig{\xnice_1}) \\ 
  &\leq \frac{C}{\MKtilde^{-1}(t)} \normbig{\xnice_1}_{D(\A)} .
 \end{align*}
Without loss of generality we may assume that $\chi=1$ on $B_\rho$. Observe that the norm of elements of $D(\A)$, supported in $\overline{\Omega}\cap B_\rho$, is equivalent to the norm in the space $H^2\times H^1(\Omega)$. This follows from maximal regularity of the Dirichlet-Laplace operator on the bounded and smooth domain $\Omega\cap B_\rho$. Thus the last inequality (restricted to those $\xnice_1$ with support in $B_\rho$) implies the conclusion of the theorem. 

(ii) The proof of the second assertion follows in exactly the same way, now using (\ref{eq:G_has_log-sing}) in addition.
\end{proof}

Let us go back to the situation described at the beginning of Section \ref{sec: Decay of the local energy}. We assume for simplicity of presentation that $d$ is odd. We see that we can apply the above theorem with $M=\delta\inv$ for some $\delta>0$ and $K(s)=C\exp(Cs^{\alpha})$ for some $\alpha>0$. Thus, we get
\begin{align*}
 p_m(t) \leq \frac{C}{t^{\frac{m}{\alpha}}} \text{ for } t\geq 1 .
\end{align*}
So our approach helps to remove the logarithmic loss in this situation. 

If $M=\delta\inv$, (at least) in some cases, it might be possible that the bound on the resolvent in (b) is given by a polynomial $K(s)=1\vee s^\alpha,s\geq0$ for some $\alpha\geq 0$. Even in this situation our result seems to be better than known results. In fact, Theorem \ref{thm:application} implies for any $c\in(0,1)$ a decay rate 
\begin{align*}
 p_m(t) \leq C e^{-cbt} \text{ for } t\geq 0,
\end{align*}
where $b=m\delta(m+\alpha-\alpha\wedge(m-1))\inv$. In case $\alpha>m-1$ the value $c=1$ is also allowed (i.e. $cb=\delta$). Therefore the obtained decay rate crucially depends on the admissible values for $c$ and on the concrete definition of $\Ktilde$. To the best of our knowledge our result gives the fastest decay rate in this situation. 

 
\begin{appendix}

\section{Proof of Proposition \ref{prp:resolvent_A_vs_R}}

 From (\ref{eq: resolvents R vs A}) we deduce that
 \begin{align}\label{eq: G vs R}
  G_{\chi}(z) = \left(
  \begin{array}{cc}
   zR_{\chi}(z)            & R_{\chi}(z)   \\
   z^2R_{\chi}(z)-\chi^2   & zR_{\chi}(z)  \\
  \end{array}\right) ,\,
  z^2R_{\chi}(z)-\chi^2 = \chi R(z) \Delta \chi .
 \end{align}
 Therefore the last statement of the proposition follows directly from
 \begin{align*}
  G_{\chi}(z)(0,g) = \left( R_{\chi}(z) g, z R_{\chi}(z) g \right).
 \end{align*}
 To prove the inequality displayed in the proposition we assume without loss of generality that $\absbig{z}\geq 1$. Furthermore we let $\chi_1$ be a function satisfying the same constraints as $\tilde{\chi}$ but with support contained in the interior of the set where $\tilde{\chi}$ is equal to $1$. Let $H^{-1}_D(\Omega)$ be the dual space of $H^1_D(\Omega)$. Clearly $\Delta:H^1_D(\Omega)\rightarrow H^{-1}_D(\Omega)$ is continuous. Furthermore the commutator $[\Delta,\chi]:H^1_{D}(\Omega)\rightarrow L^2(\Omega)$ is continuous too. This is not completely obvious since $[\Delta,\chi]=\nabla\chi\cdot\nabla + (\Delta\chi)$ has a zeroth order term. Fortunately, $\Delta \chi$ is compactly supported, $\partial\Omega\neq\emptyset, \partial\Omega\in C^\infty$ and therefore $\Delta\chi$ acts as a bounded operator on $H^1_D(\Omega)$ by the Poincar\'{e}-Steklov inequality for bounded domains. By the same reasoning we have already seen in Section \ref{sec:waves_ext_dom} that $\chi$ acts as a bounded operator on $H^1_D(\Omega)$. Before coming to the first estimates let us finally note that for all $z\in\C\backslash\R_-$ and $g\in L^2(\Omega)$ we have
 \begin{align}\label{eq: duality for R}
  R_{\chi}(z)^*g = R_{\chi}(\overline{z})g = \overline{R_{\chi}(z)\overline{g}}.
 \end{align}
 Here the bars mean the complex conjugate and $^*$ means the $L^2$-adjoint of an operator. If $z$ is a pole of $R_{\chi}$ this equality simply means that $\overline{z}$ is a pole too.
 
 Our goal is to verify the following estimates:
 \begin{align}\label{eq: R vs R I}
  \normbig{zR_{\chi}(z)}_{H^1_D\rightarrow H^1_D} &\lesssim \frac{1}{\absbig{z}} + \absbig{z}\normbig{R_{\tilde{\chi}}(z)}_{L^2\rightarrow L^2}, \\ \label{eq: R vs R II}
  \normbig{\chi R(z)\Delta\chi}_{H^1_D\rightarrow L^2} &\lesssim \frac{1}{\absbig{z}} + \absbig{z} \normbig{R_{\chi_1}(z)}_{L^2\rightarrow L^2}, \\ \label{eq: R vs R III}
  \normbig{R_{\chi}(z)}_{L^2\rightarrow H^1_D} &\lesssim \frac{1}{\absbig{z}} + \absbig{z}\normbig{R_{\chi_1}(z)}_{L^2\rightarrow L^2}.
 \end{align}
 By (\ref{eq: G vs R}) this implies the conclusion of the proposition.

 \textbf{Step 1.} Estimation of $\normbig{R_{\chi}(z)}_{L^2\rightarrow H^1_D}$. Let $f\in L^2(\Omega)$ and $u = R(z)\chi f$. Then, by Proposition \ref{prp:Burq_LaxPhillips_Vodev}, the $\Ltwoloc$-function $u$ is a distributional solution of
 \begin{align}\label{eq: aux equation R}
  \left\{
  \begin{array}{lr}
   z^2 u(x) - \Delta u(x) = \chi(x) f(x) &  (x\in\Omega), \\
   u(x) = 0 & (x\in\partial\Omega). \\
  \end{array}
  \right.
 \end{align}
 Testing the equation with $\chi \overline{u}$ leads after a short calculation, using integration by parts, to
 \begin{align*}
  \normbig{\chi \nabla u}_{L^2}^2 \lesssim \frac{1}{\absbig{z}^2} \normbig{\chi f}_{L^2}^2 + \absbig{z}^2 \normbig{(\nabla \chi)u}_{L^2}^2.
 \end{align*}
 This implies (\ref{eq: R vs R III}).

 \textbf{Step 2.} Estimation of $\normbig{\chi R(z)\Delta\chi}_{H^1_D\rightarrow L^2}$.
 \begin{align*}
  \normbig{\chi R(z)\Delta\chi}_{H^1_D\rightarrow L^2} &= \normbig{R_{\chi}(z)\Delta + \chi R(z) [\Delta, \chi]}_{H^1_D\rightarrow L^2} \\
  &\lesssim \normbig{R_{\chi}(z)}_{H^{-1}_D\rightarrow L^2} + \normbig{R_{\chi_1}(z)}_{L^2\rightarrow L^2} \\
  &\lesssim \frac{1}{\absbig{z}} + \absbig{z}\normbig{R_{\chi_1}(z)}_{L^2\rightarrow L^2}.
 \end{align*}
 From the second to the third line we used a duality argument (using (\ref{eq: duality for R})) together with (\ref{eq: R vs R III}). We have proved (\ref{eq: R vs R II}).
 
 \textbf{Step 3.} Estimation of $\normbig{zR_{\chi}(z)}_{H^1_D\rightarrow H^1_D}$. First we observe that by (\ref{eq: R vs R III})
 \begin{align*}
  \normbig{z^2 R_{\chi}(z)}_{H^1_D\rightarrow H^1_D} &= \normbig{1 + R_{\chi}(z)\Delta + \chi R(z)[\Delta,\chi]}_{H^1_D\rightarrow H^1_D} \\
  &\leq 1 + \normbig{R_{\chi}(z)\Delta}_{H^1_D\rightarrow H^1_D} + \normbig{R_{\chi_1}(z)}_{L^2\rightarrow H^1_D} \\
  &\lesssim 1 + \normbig{R_{\chi}(z)}_{H^{-1}_D\rightarrow H^1_D} + \absbig{z}\normbig{R_{\chi_1}(z)}_{L^2\rightarrow L^2} .
 \end{align*}
 It remains to estimate the middle term in the last line. Let $f\in H^{-1}_D(\Omega)$ and let $u\in H^1_D(\Omega)$ be the solution of (\ref{eq: aux equation R}) given by $R(z)\chi f$. Testing the equation with $\chi \overline{u}$ leads after a short calculation to
 \begin{align*}
  \normbig{\chi \nabla u}_{L^2}^2 \lesssim \normbig{\chi f}_{H^{-1}_D}^2 + \absbig{z}^2 \normbig{(\nabla \chi)u}_{L^2}^2.
 \end{align*}
 This implies together with a duality argument (using (\ref{eq: duality for R})) and (\ref{eq: R vs R III})
 \begin{align*}
  \normbig{R_{\chi}(z)}_{H^{-1}_D\rightarrow H^1_D} \lesssim 1 + \absbig{z} \normbig{R_{\chi_1}(z)}_{H^{-1}_D\rightarrow L^2} 
  \lesssim 1 + \absbig{z}^2 \normbig{R_{\tilde{\chi}}(z)}_{L^2\rightarrow L^2}.
 \end{align*}
 But now this in turn implies (\ref{eq: R vs R I}). The proof of Proposition \ref{prp:resolvent_A_vs_R} is finished.
\end{appendix}


\section*{Acknowledgements} I am most grateful to Ralph Chill, Yuri Tomilov and Sebastian Kr\'ol for valuable discussions on the topic of this article. I would like to thank the department of mathematics of the Nicolaus Copernicus University in Toru\'n for its hospitality. The idea to work on this topic came to me during a visit in december 2016. I am also grateful to Markus Hartlapp and Hannes Weisse for proofreading parts of the manuscript.


\hfill

Technische Universit\"{a}t Dresden, Fachrichtung Mathematik, Institut f\"{u}r Analysis, 01062, Dresden, Germany. Email: \textit{Reinhard.Stahn@tu-dresden.de}

\end{document}